\documentclass[11pt]{article}
\usepackage{amsmath,amssymb,amsthm}
\usepackage{graphicx,subfigure,float,url}
\usepackage{pdfsync}
\usepackage[english]{babel}
\usepackage[utf8]{inputenc}
\usepackage{tipa}
\usepackage{tikz}

\newlength\imagewidth
\newlength\imagescale

\usepackage{fancybox}
\usepackage{listings}
\usepackage{mathtools}
\usepackage{amsfonts}
\usepackage[bottom]{footmisc}
\usepackage{verbatim}

\usepackage{colortbl}

\usepackage{float}
\usepackage{makeidx}
\normalbaselines

\newtheorem{Theorem}{Theorem}[section]

\newtheorem{Definition}[Theorem]{Definition}

\newtheorem{Lemma}[Theorem]{Lemma}

\newtheorem{Corollary}[Theorem]{Corollary}

\newtheorem{Remark}[Theorem]{Remark}

\newcommand{\N}{\mathbb N}

\newcommand{\RR}{{{\rm I} \kern -.15em {\rm R} }}

\newcommand{\C}{{{\rm l} \kern -.42em {\rm C} }}

\newcommand{\nat}{{{\rm I} \kern -.15em {\rm N} }}

\newcommand{\dv}{\mbox{\rm div}\ }

\newcommand{\lt}{\left}
\newcommand{\rt}{\right}

\newcommand{\be}{\begin{equation}}
\newcommand{\ee}{\end{equation}}
\newcommand{\beq}{\begin{eqnarray}}
\newcommand{\eeq}{\end{eqnarray}}
\newcommand{\beqs}{\begin{eqnarray*}}
\newcommand{\eeqs}{\end{eqnarray*}}
\newcommand{\bt}{\begin{Theorem}}
\newcommand{\et}{\end{Theorem}}
\newcommand{\br}{\begin{Remark}}
\newcommand{\er}{\end{Remark}}
\newcommand{\bc}{\begin{Corollary}}
\newcommand{\ec}{\end{Corollary}}
\newcommand{\bl}{\begin{Lemma}}
\newcommand{\el}{\end{Lemma}}
\newcommand{\bd}{\begin{definition}}
\newcommand{\ed}{\end{definition}}

\renewcommand{\geq}{\geqslant}
\renewcommand{\leq}{\leqslant}

\graphicspath{{figures/}}

\title{Consensus of  the Hegselmann-Krause opinion \\formation  model with  time delay}
\author{
Young-Pil Choi\footnote{Department of Mathematics, Yonsei University, 50 Yonsei-Ro, Seodaemun-Gu, Seoul 03722, Republic of Korea (\texttt{ypchoi@yonsei.ac.kr}). }
\and
Alessandro Paolucci\footnote{Dipartimento di Ingegneria e Scienze dell'Informazione e Matematica, Universit\`{a} di L'Aquila, Via Vetoio, Loc. Coppito, 67010 L'Aquila Italy (\texttt{alessandro.paolucci2@graduate.univaq.it}).}
\and
Cristina Pignotti\footnote{Dipartimento di Ingegneria e Scienze dell'Informazione e Matematica, Universit\`{a} di L'Aquila, Via Vetoio, Loc. Coppito, 67010 L'Aquila Italy (\texttt{pignotti@univaq.it}).}
}

\date{}

\begin{document}

\textwidth=160 mm

\textheight=225mm

\parindent=8mm

\frenchspacing

\maketitle

\begin{abstract}
In this paper, we study Hegselmann--Krause models with a time-variable time delay. Under appropriate assumptions, we show the exponential asymptotic consensus when the time delay satisfies a suitable smallness assumption. Our main strategies for this are based on Lyapunov functional approach and careful estimates on the trajectories. We then study the mean-field limit from the many-individual Hegselmann--Krause equation to the continuity-type partial differential equation as the number $N$ of individuals goes to infinity. For the limiting equation, we prove global-in-time existence and uniqueness of measure-valued solutions. We also use the fact that constants appearing in the consensus estimates for the particle system are independent of $N$ to extend the exponential consensus result to the continuum model. Finally, some numerical tests are illustrated.
\end{abstract}

\vspace{5 mm}

\def\qed{\hbox{\hskip 6pt\vrule width6pt
height7pt
depth1pt  \hskip1pt}\bigskip}

%% {\bf 2000 Mathematics Subject Classification:}
%%35L05, 93D15

 %%{\bf Keywords and Phrases:}  wave equation,  delay feedbacks, stabilization

\section{Introduction}
\label{pbform}
%\hspace{5mm}

\setcounter{equation}{0}

Recently, multi-agent models have attracted the interest of many authors in several scientific disciplines. A particularly interesting aspect of the dynamics of multi-agent systems is the natural self-organization which leads to the emergence of a globally collective behavior. This happens for biological systems \cite{Cama, CS1}, physical systems \cite{Stro}, ecosystems \cite{Sole}, social sciences \cite{Bellomo, BN, Castellano, Campi}. We also mention engineering applications \cite{Bullo, Desai, Jad}, economics models \cite{HK, Marsan}, control problems \cite{ Aydogdu, Piccoli, W}.
Various models have been proposed to study opinion dynamics \cite{Castellano, During, HK, Lo,  XWX}.

It is also natural to include a time delay in the model. Time delay effects frequently appear in applications, in biological and physical models, as well as in social dynamics, economics and control problems. Indeed, a certain time is needed for each agent to receive information from other agents. Also, a time lag can appear in the evolution of a multi-agent system as a reaction time. In the current work, we are interested in the so--called Hegselmann--Krause model (see \cite{HK}) and study the opinion formation in the presence of a time-variable time delay, see also \cite{Lu} for a model with constant interaction rates and constant delay. Among other results about Hegselmann--Krause type models without time delays, we refer to \cite{CF, JM}, where clusters formation is studied in case of bounded confidence. For second--order consensus models, in particular Cucker--Smale type models with time delays, we refer to recent papers \cite{CH, CL, CP, EHS,  HM, PR,  PT}.

Let $x_i \in \RR^{d}$ represent an opinion of $i$-th agent, then our main system reads as
\begin{equation}
\label{modello}
\frac{dx_i(t)}{dt}=\frac{\lambda} N\sum_{j\neq i} a_{ij}(t) (x_j(t-\tau(t))-x_i(t)),\quad i=1,\dots, N, \quad t > 0,
\end{equation}
with initial data given by
\begin{equation}
\label{modello,iniziale}
x_i(s)=x_{i,0}(s), \quad\forall\, s\in [-\tau(0),0],
\end{equation}
where the coupling strength $\lambda$ is a positive parameter and the communication rates $a_{ij}(t)$ are of the form
\begin{equation}\label{symmetric}
a_{ij}(t)=\psi(|x_j(t-\tau(t))-x_i(t)|),
\end{equation}
with $\psi: [0, +\infty)\rightarrow (0, +\infty)$ non-increasing function or, in the spirit of \cite{MT},
\begin{equation}\label{notsymmetric}
a_{ij}(t)= \frac {N \psi(|x_j(t-\tau(t))-x_i(t)|)} {\sum_{k=1}^N \psi(|x_k(t-\tau(t))-x_i(t)|)}
\end{equation}
with $\psi$ as above. Without loss of generality, we assume $\psi (0)=1.$ Note that in both cases we have, for $i=1, \dots, N,$
\begin{equation}\label{C2}
\frac 1 N \sum_{j=1}^N a_{ij} (t) \leq 1, \quad \forall \, t\geq 0.
\end{equation}
The time delay $\tau(\cdot)$ is a positive function belonging to $W^{1,\infty}_{loc}(0,+\infty)$ and satisfying
\begin{equation}\label{tau1}
0\le \tau (t)\le \overline \tau, \quad \forall \, t>0,
\end{equation}
and
\begin{equation}\label{tau2}
\tau^\prime(t)\le c, \quad c\in (0,1), \quad \forall \, t>0.
\end{equation}
Now, we define the position diameter as follows:
\begin{equation}\label{diameter}
d_X(t):=\max_{1 \leq i,j \leq N} |x_i(t)-x_j(t)|.
\end{equation}

\begin{Definition}
{\rm We say that the dynamics, subject to \eqref{modello}, converges to \emph{consensus} if}
\[
d_X(t)\rightarrow 0 \quad \mbox{\rm as} \quad t\rightarrow +\infty.
\]
\end{Definition}
By constructing a suitable Lyapunov functional, in the first part of our work, we will deduce exponential consensus estimates when the time delay is sufficiently small. Note that the constants in the decay estimates are independent of the number of the agents; this is crucial in order to extend the consensus estimate for the associated continuum equation below. 

In the second part of the paper, we will study the continuum model obtained as mean-field limit of the particle system
when $N\rightarrow \infty.$ Let $\mathcal{M}(\RR^{d})$ be the set of probability measures on the space $\RR^{d}$. Then, the continuum model associated to the particle system \eqref{modello} is given by
\begin{align}\label{kinetic model}
\begin{aligned}
&\partial_t \mu_t+ \dv (F[\mu_{t-\tau(t)}]\mu_t)=0, \quad x\in \RR^d, \quad t>0, \cr
&\mu_s=:g_s, \quad x\in \RR^d, \quad  s\in[-\tau(0),0],
\end{aligned}
\end{align}
where the velocity field $F$ is given by either
\begin{equation}
\label{F-symmetric}
F[\mu_{t-\tau(t)}](x)=\int_{\RR^d} \psi(|x-y|)(y-x)\,d\mu_{t-\tau(t)}(y),
\end{equation}
or
\begin{equation}
\label{F-nonsymmetric}
F[\mu_{t-\tau(t)}](x)= \displaystyle{\frac{\int_{\RR^d} \psi(|x-y|)(y-x)\,d\mu_{t-\tau(t)}(y)}{\int_{\RR^d} \psi(|x-y|)\,d\mu_{t-\tau(t)}(y)},}
\end{equation}
according to \eqref{symmetric} and \eqref{notsymmetric}, and $g_s \in \mathcal{C}([-\tau(0),0];\mathcal{M}(\RR^d))$.
For the derivation of the  continuum model associated to the Hegselmann--Krause system, without time delays, we refer to \cite{CFT}. Since then, kinetic formulations for opinion dynamics have been the objects of several works, see e.g \cite{CCR, CFT2, Carrillo,  Piccoli}.
On the other hand, the continuum formulation in presence of delay effects seems new.

Following a similar strategy to the one used for the kinetic Cucker-Smale equation in  \cite{CH} (see also \cite{CP}) we will study the global-in-time well--posedness of  \eqref{kinetic model}, more precisely, global existence and uniqueness of measure-valued solutions. In addition, we prove a stability estimate which allows to rigorously justify the mean--field limit procedure. Furthermore, the stability estimate, together with the uniform-in-$N$ consensus estimate for the particle model, enables us to extend  the asymptotic  consensus  to the continuum equation.

The rest of this paper is organized as follows. In Sect. \ref{ODE}, we will study the particle model \eqref{modello} and show the consensus behavior for small delays. In Sect. \ref{PDE}, we will focus on the continuum model \eqref{kinetic model} obtained formally from the particle system, and we will analyze it in the set of the probability measures employing the Wasserstein distance of order $p \in [1,\infty]$. Finally, some numerical simulations are illustrated in Sect. \ref{numerics}.

\section{Exponential consensus behavior of the particle system }\label{ODE}
In order to study the convergence to consensus of system \eqref{modello}-\eqref{modello,iniziale}, we need some auxiliary lemmas.
\begin{Lemma}\label{Lemma iniziale}
Let $\{ x_i\}_{i=1}^N$ be a solution to \eqref{modello}-\eqref{modello,iniziale}. Suppose that the initial opinions $x_{i,0}(s)$ are continuous functions of $s\in [-\tau (0) ,0]$. Moreover, let us denote
\[
R=\max_{s\in [-\tau(0),0]} \max_{1 \leq i \leq N} |x_i(s)|.
\]
Then we have
\[
R_X(t):= \max_{1 \leq i \leq N} |x_i(t)| \leq R, \quad \forall t\in [-\tau (0), +\infty ).
\]
\end{Lemma}
\begin{proof}
For any $\epsilon>0$, let us define
$$
S^{\epsilon}=\lt\{ t>0 \,:\  R_X(s)<R+\epsilon , \ \  s\in [0,t) \rt\}.
$$
Let us denote $T^\epsilon_*:=\sup S^\epsilon .$
By continuity, $S^{\epsilon} \neq \emptyset$. Hence, $T_*>0$.
We claim that $T_*^{\epsilon}=+\infty.$ To prove this we argue by contradiction. Suppose that
$T_*^{\epsilon}<+\infty.$ Then,
\begin{equation}\label{C1}
\lim_{t\rightarrow T_*^{\epsilon-}} R_X(t)=R+\epsilon  \quad \text{and} \quad R_X(t)<R+\epsilon ,\quad \forall t<T_*^{\epsilon}.
\end{equation}
Now, for $t<T_*^{\epsilon}$, we can compute
\begin{align*}
\frac{1}{2}\frac{d}{dt}|x_i(t)|^2 &= \lt\langle x_i, \frac \lambda N\sum_{j\neq i} a_{ij}(t) (x_j(t-\tau (t)) -x_i(t)) \rt\rangle \\
&= \frac \lambda N\sum_{j\neq i} a_{ij}(t) \Bigl( \langle x_i(t), x_j(t-\tau(t)) \rangle -|x_i(t)|^2 \Bigr) \\
&\leq  \frac \lambda N\sum_{j\neq i} a_{ij}(t) |x_i(t)| \Bigl( |x_j(t-\tau(t))|-|x_i(t)| \Bigr) \\
&\leq  \frac \lambda N\sum_{j\neq i} a_{ij}(t) \Bigl( R+{\epsilon}-|x_i(t)| \Bigr) |x_i(t)|.
\end{align*}
Then, we deduce
$$
\frac{d}{dt}|x_i(t)|\le \frac \lambda N\sum_{j\neq i} a_{ij}(t) \Bigl( R+{\epsilon}-|x_i(t)| \Bigr),
$$
and so, recalling \eqref{C2}, we have
\[
\frac{d}{dt} |x_i(t)| \leq \lambda (R+\epsilon -|x_i(t)|).
\]
Hence, Gronwall's inequality yields the estimate
\begin{equation}\label{C4}
|x_i(t)|\leq e^{-\lambda t} \Bigl( \vert x_i(0)\vert -R -\epsilon  \Bigr)+R+\epsilon <R+\epsilon.
\end{equation}
From \eqref{C4}
we deduce $\lim_{t\rightarrow T_*^{\epsilon}} R_X(t)<R+\epsilon$, which is in contradiction with \eqref{C1}.
Being $\epsilon$ arbitrary, the lemma is proved.
\end{proof}
\begin{Remark}{\rm
From Lemma \ref{Lemma iniziale}, we have that for $i,j=1, \dots, N,$
$$
|x_j(t-\tau(t))-x_i(t)|\leq |x_j(t-\tau(t))|+|x_i(t)|\leq 2R,\quad \forall\, t\ge 0.
$$
So, from \eqref{symmetric} and \eqref{notsymmetric}, we deduce
\begin{equation}\label{stimapsi}
a_{ij}(t)\geq \psi (2R),\quad \forall\ t\ge 0.
\end{equation}
}
\end{Remark}
The diameter function $d_X$  is not differentiable  in general. Thus, we introduce the upper Dini derivative to consider the time derivative of this function. For a given continuous function $F = F(t)$, the upper Dini derivative of $F$ at $t$ is defined by
$$
D^+ F(t) := \limsup_{h \to 0^+} \frac{F(t+h) - F(t)}{h}.
$$
Note that the Dini derivative coincides with the usual derivative when the function is differentiable at $t$.

Then we have the following lemma.
\begin{Lemma}\label{stimaresto}
Let $\{ x_i\}_{i=1}^N$ be a solution to \eqref{modello}-\eqref{modello,iniziale}. If
\[
\sigma_{\tau}(t):= \int_{t-\tau (t)}^t \max_{1 \leq k \leq N} |\dot{x}_k(s)|\,ds,
\]
then the position diameter $d_X(\cdot)$ satisfies
\[
D^+ d_X(t)\leq \frac{2 \lambda }{\psi (2R)} \sigma_{\tau} (t)-\lambda \psi (2R)d_X(t), \quad \forall \, t \geq 0.
\]
\end{Lemma}
\begin{proof}
Due to the continuity of the trajectories $x_i(t),$ $i=1, \dots, N,$  there is an at most countable system of open disjoint
intervals $\{\mathcal{I}_\delta\}_{\delta\in\N}$ such that
$$
   \bigcup_{\delta\in\N} \overline{\mathcal{I}_\delta} = [0,\infty)
$$
and  for each ${\delta\in\N}$ there exist indices $i(\delta)$, $j(\delta)$
such that
$$
   d_X(t) = |x_{i(\delta)}(t) - x_{j(\delta)}(t)|, \quad t\in \mathcal{I}_\delta.
$$
For simplicity of notation, we can  put  $i:=i(\delta)$, $j:=j(\delta).$ Of course, we can assume $i\ne j.$ For $t\in \mathcal{I}_\delta$, we have

\begin{equation}\label{Q1}
\begin{array}{l}
\displaystyle{
\frac{1}{2}D^+ d_X^2(t) }\\
\hspace{0.5 cm}\displaystyle{
=\frac \lambda N \Big \langle x_i(t) -x_j(t), \sum_{k\neq i} a_{ik}(t) (x_k(t-\tau(t)) -x_i(t))-\sum_{k\neq j} a_{jk}(t) (x_k(t-\tau (t)) -x_j(t)) \Big\rangle } \\
\displaystyle{
\hspace {0.5 cm}
 = \frac \lambda N \sum_{k\neq i} a_{ik}(t) \langle x_i(t)-x_j(t), x_k(t-\tau(t)) -x_i(t) \rangle }\\
 \hspace{2.5 cm}\displaystyle{
 -\frac \lambda N\sum_{k\neq j} a_{jk} (t)\langle x_i(t)-x_j(t), x_k(t-\tau(t)) -x_j(t) \rangle } \\
\hspace{0.5 cm}=: I_1+I_2.
\end{array}
\end{equation}
Now, we can rewrite $I_1$ and $I_2$ as
\begin{equation}\label{I1}
\begin{array}{l}
\displaystyle{
I_1= \frac \lambda N\sum_{k\neq i} a_{ik}(t) \langle x_i(t)-x_j(t), x_k(t-\tau(t)) -x_k(t)\rangle}\\
\displaystyle{
\hspace{1 cm} +\frac \lambda N \sum_{k\neq i} a_{ik} (t) \langle x_i(t)-x_j(t),x_k(t)-x_i(t)\rangle} \quad \mbox{and}
\end{array}
\end{equation}
\begin{equation*}
\begin{array}{l}
\displaystyle{
I_2= -\frac \lambda N\sum_{k\neq j} a_{jk}(t) \langle x_i(t)-x_j(t), x_k(t-\tau(t)) -x_k(t)\rangle}\\
\displaystyle{
\hspace{1 cm} -\frac \lambda N \sum_{k\neq j} a_{jk} (t) \langle x_i(t)-x_j(t),x_k(t)-x_j(t)\rangle,}
\end{array}
\end{equation*}
respectively. We observe that for all$k=1, \dots, N,$
\begin{equation*}
\langle x_i(t)-x_j(t),x_k(t)-x_i(t)\rangle \leq 0.
\end{equation*}
Indeed, using Cauchy-Schwartz inequality, we have that
\begin{align*}
\langle x_i(t)-x_j(t) , x_k(t)-x_i(t) \rangle &= \langle x_i(t)-x_j(t),x_k(t)-x_j(t)\rangle -|x_i(t)-x_j(t)|^2 \\
&\leq  |x_i(t)-x_j(t)| \left (|x_k(t)-x_j(t)|-|x_i(t)-x_j(t)|\right )\\
& \leq  0.
\end{align*}
Now, observe that in both cases \eqref{symmetric} and \eqref{notsymmetric},
\begin{equation}\label{stimapsialto}
 a_{ik}(t)\leq \frac{1}{\psi(2R)},\quad \forall \ t\ge 0.
\end{equation}
Hence, using \eqref{stimapsi} and \eqref{stimapsialto} in \eqref{I1}, we obtain
\begin{equation}\label{stimaI1}
I_1 \leq \frac \lambda N \frac {d_X(t)} {\psi (2R)} \sum_{k=1}^N |x_k(t-\tau(t)) -x_k(t)|+\frac \lambda N {\psi (2R)} \sum_{k=1}^N \langle x_i(t)-x_j(t),x_k(t)-x_i(t)\rangle .
\end{equation}
Now, observe that
\begin{equation*}
-\langle x_i(t)-x_j(t),x_k(t)-x_j(t)\rangle \leq 0, \quad \forall\ k=1, \dots, N.
\end{equation*}
Then, arguing analogously to before, one can estimate
\begin{equation}\label{stimaI2}
I_2 \leq \frac \lambda N \frac{d_X(t)}{\psi (2R)}\sum_{k=1}^N|x_k(t-\tau(t)) -x_k(t)|-\frac \lambda N \psi (2R)\sum_{k=1}^N \langle x_i(t)-x_j(t),x_k(t)-x_j(t)\rangle.
\end{equation}
Therefore, using \eqref{stimaI1} and \eqref{stimaI2} in \eqref{Q1}, we obtain
$$
\frac 1 2 D^+ d_X^2(t)\leq 2 \frac{\lambda}N \frac {d_X(t)}{\psi (2R)} \sum_{k=1}^N |x_k(t-\tau(t))-x_k(t)| -\lambda \psi (2R)d_X^2(t),
$$
and so
\begin{equation}\label{Q2}
D^+ d_X(t)\leq 2 \frac{\lambda}{ N\psi (2R)} \sum_{k=1}^N |x_k(t-\tau(t))-x_k(t)| -\lambda \psi (2R)d_X(t).
\end{equation}
Noticing that
\begin{align*}
\sum_{k=1}^N |x_k(t-\tau(t)) -x_k(t)|&\leq  \sum_{k=1}^N \int_{t-\tau(t) }^t |\dot{x}_k(s)| \,ds \\
&\leq  N\sigma _{\tau} (t),
\end{align*}
then from \eqref{Q2} we obtain
$$
D^+d_X(t) \leq \frac{2 \lambda }{\psi (2R)} \sigma_{\tau} (t)-\lambda \psi (2R) d_X(t),
$$
which proves the lemma.
\end{proof}
\begin{Lemma}
Let $\{ x_i\}_ {i=1}^N$ be a solution to \eqref{modello}-\eqref{modello,iniziale}. Then, we have
\begin{equation}\label{Q3}
\max_{1 \leq i \leq N} |\dot{x}_i(t)| \leq \frac{\lambda }{\psi (2R)}\sigma_\tau (t)+\frac{\lambda}{\psi (2R)}d_X(t).
\end{equation}
\end{Lemma}
\begin{proof}
A straightforward computation gives
\begin{align*}
|\dot{x}_i(t)| &= \lt| \frac \lambda N \sum_{j\neq i} a_{ij}(t) (x_j(t-\tau(t)) -x_j(t))+\frac \lambda N\sum_{j\neq i} a_{ij}(t) (x_j(t) -x_i(t)) \rt| \\
&\leq \frac \lambda N\sum_{j\neq i } a_{ij}(t) |x_j(t-\tau(t)) -x_j(t)|+\frac \lambda N\sum_{j\neq i} a_{ij}(t) |x_j(t)-x_i(t)|\\
&\leq  \frac{\lambda}{N\psi (2R)} \sum_{j\neq i} \int_{t-\tau (t)}^t |\dot{x}_j(s)|\,ds +\lambda \frac{N-1}{N\psi (2R)}d_X(t)\\
&\leq  \lambda \frac{N-1}{N\psi (2R)} \sigma _{\tau} (t)+\lambda \frac{N-1}{N\psi (2R)}d_X(t).
\end{align*}
Taking the maximum for $i\in \{1, \dots, N\},$ we obtain \eqref{Q3}.
\end{proof}

\begin{Theorem}\label{flocking particle}
Let $\{ x_i\}_{i=1}^N$ be a solution to \eqref{modello}-\eqref{modello,iniziale}. Suppose that
\[
\overline \tau <\ln \left (
1+ \frac {\psi^3 (2R)} {2 + \psi^2 (2R)} \, \frac{1-c} \lambda
\right ).
\]
Then, there exist $\gamma, C >0 $ such that
\[
d_X(t) \leq C e^{-\gamma t}, \quad \forall \,t\geq 0.
\]
\end{Theorem}
\begin{proof}
Consider the Lyapunov functional
\[
F(t)=d_X(t)+\beta \int_{t-\tau(t) }^t {e^{-(t-s)} \int_s^t \max_{1 \leq j \leq N} |\dot{x}_j(\sigma )|\,d\sigma} ds.
\]
Then, we have
\begin{align*}
D^+F(t) &= D^+ d_X(t)-\beta (1-\tau '(t)) e^{-\tau (t)} \int_{t-\tau (t)}^t \max_{1 \leq j \leq N} |\dot{x}_j(\sigma)|\,d\sigma \cr
&\quad +\beta \int_{t-\tau (t)}^t e^{-(t-s)} \max_{1 \leq j \leq N} |\dot{x}_j(t)|\,ds \cr
&\quad - \beta \int_{t-\tau (t)}^t e^{-(t-s)}\int_s^t \max_{1 \leq j \leq N} |\dot{x}_j(\sigma )|\,d\sigma ds  \cr
&\leq  \lt( -\lambda \psi(2R) + (1-e^{-\tau (t)})\frac{\lambda \beta}{\psi (2R)}\rt) d_X(t)\\
&\quad +\lt( \frac{2\lambda }{\psi(2R)}-\beta(1-\tau '(t))e^{-\tau (t)}+\beta (1-e^{-\tau(t)})\frac{\lambda}{\psi (2R)} \rt) \sigma_{\tau} (t) \\
&\quad - \beta \int_{t-\tau (t)}^t e^{-(t-s)} \int_s^t \max_{1 \leq j \leq N} |\dot{x}_j(\sigma)|\,d\sigma ds.
\end{align*}
Hence, from the assumptions \eqref{tau1} and \eqref{tau2} we deduce
\begin{align}\label{Q4}
D^+F(t) &\leq \lt( -\lambda \psi(2R) + (1-e^{-\overline \tau})\frac{\lambda \beta}{\psi (2R)}\rt) d_X(t)\cr
&\quad + \lt( \frac{2\lambda }{\psi(2R)}-\beta(1-c)e^{-\overline \tau}+\beta (1-e^{-\overline \tau})\frac{\lambda}{\psi (2R)} \rt) \sigma_{\tau} (t)\cr
&\quad - \beta \int_{t-\tau (t)}^t e^{-(t-s)} \int_s^t \max_{1 \leq j \leq N} |\dot{x}_j(\sigma)|\,d\sigma ds.
\end{align}
Now, we want to show that for $\overline \tau$ sufficiently small we can choose the positive parameter $\beta$ in the definition of  the Lyapunov functional $F(\cdot)$ such that \eqref{Q4} implies
\begin{equation}\label{Q5}
D^+F(t)\le - \gamma F(t), \quad \forall \,t\geq 0,
\end{equation}
for a suitable positive constant $\gamma.$
In order to have \eqref{Q5} the following two conditions have to be satisfied:
\begin{equation}\label{beta1}
\frac{2\lambda }{\psi(2R)}-\beta(1-c)e^{-\overline \tau}+\beta (1-e^{-\overline \tau})\frac{\lambda}{\psi (2R)}\le 0 \quad \mbox{and}
\end{equation}
\begin{equation}\label{beta2}
-\psi(2R) + (1-e^{-\overline \tau})\frac{\beta}{\psi (2R)}<0.
\end{equation}
We can rewrite \eqref{beta1} as
$$
{2\lambda }-\beta \psi (2R) (1-c)e^{-\overline \tau}+\lambda \beta (1-e^{-\overline \tau})\le 0,
$$
which is satisfied for
\begin{equation}\label{firstcond}
\beta \ge \frac {2\lambda }{\psi (2R) (1-c) e^{-\overline \tau }-\lambda (1-e^{-\overline\tau })}.
\end{equation}
This requires a first restriction on the time delay size, i.e.
\begin{equation}\label{tausmall1}
\overline\tau <\ln \left ( 1+ \psi (2R) \frac{(1-c)} \lambda
\right ).
\end{equation}
Condition \eqref{beta2} instead implies
\begin{equation}\label{secondcond}
\beta <\frac {\psi^2(2R)}{1-e^{-\overline\tau }}.
\end{equation}
Then, for the existence of a parameter $\beta$ satisfying both \eqref{firstcond} and \eqref{secondcond}, we need
$$
\frac {2\lambda }{\psi (2R) (1-c) e^{-\overline \tau }-\lambda (1-e^{-\overline\tau })}< \frac {\psi^2(2R)}{1-e^{-\overline\tau }},
$$
and this gives a further condition on $\overline\tau,$ namely
$$ \overline \tau <\ln \left (
1+ \frac {\psi^3 (2R)} {2 + \psi^2 (2R)} \, \frac{1-c} \lambda
\right ),
$$
which clearly implies \eqref{tausmall1}. Hence, the theorem is proved.
\end{proof}

\begin{Remark}{\rm
Note that in the case of potential communication rates as in \eqref{symmetric} the estimate \eqref{stimapsialto} can be replaced by
$$a_{ij}(t)\le 1.$$
Therefore, in such a case we can obtain an improved bound on the size of the time delay:
$$\overline \tau <\ln \left (1 +\frac {\psi}{2+\psi}\frac {1-c}{\lambda }\right ).$$}
\end{Remark}
\section{The continuum model: measure-valued solutions \& consensus behavior}\label{PDE}
In this section, in the same spirit of \cite{CH}, we provide the existence, uniqueness of solution to the continuum model \eqref{kinetic model} associated to \eqref{modello} and its consensus behavior under a suitable smallness assumption on the delay function $\tau(t)$.

In order to study the existence and uniqueness of the solution of the continuum model, we  assume that the delay function $\tau(\cdot)$ is bounded from below, namely there exists a strictly positive constant $\tau^*>0$ such that
\[
\tau(t)\geq \tau^*, \quad \forall \,t\geq 0.
\]
Moreover, we assume that the potential $\psi(\cdot)$ in \eqref{symmetric} and \eqref{notsymmetric} is also Lipschitz continuous, and we denote by $L$ its Lipschitz constant.

We define the Wasserstein distance as follows.
\begin{Definition}
Let $\mu, \nu \in \mathcal{M}(\RR^d)$ be two probability measures on $\RR^d$. Then, we define the Wasserstein distance of order $1 \leq p < \infty$ between $\mu$ and $\nu$ as
\begin{equation*}
d_p(\mu,\nu):=\inf_{\pi \in \Pi(\mu,\nu)} \lt(\int_{\RR^d \times \RR^d} |x-y|^p\,d\pi(x,y)\rt)^{1/p}
\end{equation*}
and for $p = \infty$, limiting case as $p \to \infty$,
\[
d_\infty(\mu,\nu) := \inf_{\pi \in \Pi(\mu,\nu)} \lt(\sup_{(x,y) \in \mbox\small{supp}(\pi)} | x - y| \rt),
\]
where $\Pi(\mu,\nu)$ is the set of all probability measures on $\RR^{2d}$ with marginals $\mu$ and $\nu$ (also called \emph{couplings} for $\mu$ and $\nu$), namely
\begin{equation*}
\int_{\RR^d \times \RR^d} \phi(x) \,d\pi(x,y)=\int_{\RR^d} \phi(x)\,d\mu(x), \qquad \int_{\RR^d \times \RR^d} \phi(y) \,d\pi(x,y) =\int_{\RR^d} \phi(y)\,d\nu(y),
\end{equation*}
for all continuous and bounded functions $\phi \in \mathcal{C}_b(\RR^d)$.
\end{Definition}
Note that $\mathcal{P}_p(\RR^d)$, which stands for the set of probability measures with bounded moments of order $p \in [1,\infty)$, endowed with the $p$-Wasserstein distance $d_p$ is a complete metric space. Moreover, we recall the definiton of the push-forward of a measure:
\begin{Definition}
Let $\mu$ be a Borel measure on $\RR^d$ and let $\mathcal{T}:\RR^d\rightarrow \RR^d$ be a measurable map. Then, we define the push-forward of $\mu$ via $\mathcal{T}$ as the measure given by
\begin{equation*}
\mathcal{T} \# \mu(B):=\mu (\mathcal{T}^{-1}(B)),
\end{equation*}
for all Borel sets $B\subset \RR^d$.
\end{Definition}
Furthermore, we define the notion of measure-valued solution to \eqref{kinetic model}.
\begin{Definition}
Let $T>0$ be any given time. We say that $\mu_t\in \mathcal{C}([0,T);\mathcal{M}(\RR^d))$ is a \emph{measure-valued solution} to \eqref{kinetic model} on the time-interval $[0,T)$ if for all $\phi \in \mathcal{C}_c^{\infty} (\RR^d\times [0,T))$ we have that
\begin{equation}
\label{weak form}
\int_0^T\int_{\RR^d}  (\partial_t \phi+F[\mu_{t-\tau(t)}](x)\cdot \nabla_x \phi)\,d\mu_t(x) dt+\int_{\RR^d} \phi(x,0)\,dg_0(x)=0,
\end{equation}
where $F[\mu_{t-\tau(t)}]$ is defined as in \eqref{F-symmetric} or \eqref{F-nonsymmetric}.
\end{Definition}
Let us denote $B^d(0,R)$ the ball of radius $R$ in $\RR^d$ centered at the origin.
In order to prove  existence and uniqueness of solution to the kinetic model \eqref{kinetic model}, we have the following lemma.
\begin{Lemma}
\label{Lipschitz lemma}
Let $\mu_t\in \mathcal{C}([0,T];\mathcal{M}(\RR^d))$ have uniform compact support, i.e.
$$
supp \ \mu_t \subset B^d(0,R),\quad \forall\, t\in [0,T],
$$
for some positive constant $R>0$. Then there exists a constant $K>0$ such that
\begin{equation}
\label{Lipschitz}
|F[\mu_{t-\tau(t)}](x)-F[\mu_{t-\tau(t)}](\tilde{x})|\leq K |x-\tilde{x}|,
\end{equation}
for all $x,\tilde{x}\in B^d(0,R)$ and for all $t\in[0,T]$.

Moreover, there exists a constant $C>0$ such that
\begin{equation}
\label{stimaAlto}
|F[\mu_{t-\tau(t)}](x)|\leq C,
\end{equation}
for all $x\in B^d(0,R)$ and for all $t\in [0,T]$.
\end{Lemma}
\begin{proof}
In order to prove \eqref{Lipschitz} and \eqref{stimaAlto} we have to distinguish two cases, corresponding to
 $F$ as in \eqref{F-symmetric} or \eqref{F-nonsymmetric}. \newline

 \noindent {\sl Case I ($F$ as in \eqref{F-symmetric})}:
For  any $x,\tilde{x}\in B^d(0,R)$, we have
\begin{align*}
&|F[\mu_{t-\tau(t)}](x)-F[\mu_{t-\tau(t)}](\tilde{x})| \cr
&\quad =\left\vert \int_{\RR^d}\psi (\vert x-y\vert )(y-x) \,d\mu_{t-\tau (t)} (y)-
\int_{\RR^d}\psi (\vert \tilde x-y\vert )(y-\tilde x) \,d\mu_{t-\tau (t)} (y)\right\vert \cr 
&\quad \le
\left\vert \int_{\RR^d} [\psi (\vert x-y\vert )-\psi (\vert \tilde x -y\vert )]\, y\, d\mu_{t-\tau (t)}(y)\right\vert \cr
&\quad \quad  + \left\vert \int_{\RR^d}\psi (\vert x-y\vert )\, x\, d\mu_{t-\tau (t)} (y)
- \int_{\RR^d}\psi (\vert \tilde x-y\vert )\, \tilde x\, d\mu_{t-\tau (t)} (y)\right\vert \cr
& \quad \le  R \, L \vert x-\tilde x\vert
+\left\vert
\int_{\RR^d} [\psi (\vert x-y\vert )-\psi (\vert \tilde x -y\vert )]\, x\, d\mu_{t-\tau (t)}(y)
\right\vert +\vert x-\tilde x\vert \cr
&\quad \le (1+2RL) |x-\tilde{x}|,
\end{align*}
where we have used that $\mu_t$ is a probability measure with support in $B^d(0, R)$ and $\psi (0)=1.$
Then, \eqref{Lipschitz} is proved.
Moreover, from \eqref{F-symmetric}, immediately follows $|F[\mu_{t-\tau(t)}](x)|\leq 2R$ which gives \eqref{stimaAlto}. \newline

\noindent {\sl Case II ($F$ as in \eqref{F-nonsymmetric})}:  Set
 $$\psi^*:= \inf_{y\in B^d(0,2R)} \psi(|y|) >0\,.$$
Then we estimate
\begin{align*}
&|F[\mu_{t-\tau(t)}](x)-F[\mu_{t-\tau(t)}](\tilde{x})| \cr
&\quad \leq \left| \frac{\int_{\RR^d} \psi(\vert x-y\vert ) y\,d\mu_{t-\tau(t)}(y)}{\int_{\RR^d} \psi(\vert x-y\vert ) \,d\mu_{t-\tau(t)}(y)}- \frac{\int_{\RR^d} \psi(\vert \tilde{x}-y \vert ) y\,d\mu_{t-\tau(t)}(y)}{\int_{\RR^d} \psi(\vert \tilde{x}-y\vert ) \,d\mu_{t-\tau(t)}(y)} \right| +|x-\tilde{x}|\cr
&\quad  \leq \frac{\left| \int_{\RR^d} \psi(\vert x-y\vert ) y\,d\mu_{t-\tau(t)}(y) - \int_{\RR^d} \psi(\vert \tilde{x}-y \vert ) y\,d\mu_{t-\tau(t)}(y) \right|}{\left| \int_{\RR^d} \psi(\vert x-y\vert ) \,d\mu_{t-\tau(t)}(y)\right| }\cr
&\qquad + \frac{\left|  \int_{\RR^d} \psi(\vert \tilde{x}-y \vert ) y\,d\mu_{t-\tau(t)}(y)   \right| }{\left|  \int_{\RR^d} \psi(\vert x-y\vert ) \,d\mu_{t-\tau(t)}(y)  \right| \left|  \int_{\RR^d} \psi(\vert \tilde{x}-y\vert )\, d\mu_{t-\tau(t)}(y)   \right|} \cr
&\hspace{2cm} \times \left| \int_{\RR^d} \psi(\vert x-y\vert ) \,d\mu_{t-\tau(t)}(y) - \int_{\RR^d} \psi(\vert \tilde{x}-y\vert ) \,d\mu_{t-\tau(t)}(y) \right| +\vert x-\tilde{x}\vert\cr
&\quad \leq  \left( \frac{2RL}{(\psi^*)^2}+1\right) \vert x-\tilde{x}\vert,
\end{align*}
where we used the fact that $\mu_t$ is compactly supported in $B^d(0,R)$,
$$
\left| \int_{\RR^d} \psi(\vert x-y\vert ) y\,d\mu_{t-\tau(t)}(y) - \int_{\RR^d} \psi(\vert \tilde{x}-y \vert ) y\,d\mu_{t-\tau(t)}(y) \right| \leq RL |x-\tilde{x}|,
$$
and
$$
\left| \int_{\RR^d} \psi(\vert x-y\vert ) \,d\mu_{t-\tau(t)}(y) - \int_{\RR^d} \psi(\vert \tilde{x}-y\vert ) \,d\mu_{t-\tau(t)}(y) \right|\leq L|x-\tilde{x}|.
$$
Finally, from \eqref{F-nonsymmetric} we obtain immediately
$$
|F[\mu_{t-\tau(t)}](x)|\leq R\left(\frac{1}{\psi^*}+1\right),
$$
for all $x\in B^d(0,R)$ and for all $t\in[0,T]$. Thus, we obtain \eqref{stimaAlto}.
\end{proof}
Now, we can prove the following theorem.
\begin{Theorem}
\label{teorema esistenza}
Consider the kinetic model \eqref{kinetic model}, with $g(t)\in \mathcal{C}([-\tau(0),0];\mathcal{M}(\RR^d))$, and suppose that there exists a constant $R>0$ such that
\begin{equation*}
supp \ g_t \subset B^d(0,R),
\end{equation*}
for all $t\in [-\tau(0),0]$. Then for any $T>0$ there exists a unique measure-valued solution $\mu_t\in \mathcal{C}([0,T);\mathcal{M}(\RR^d))$ of \eqref{kinetic model} in the sense of \eqref{weak form}. Moreover, $\mu_t$ is uniformly compactly supported in position and we have that
\begin{equation}
\label{push-forward mu}
\mu_t=X(t;\cdot) \# \mu_0,
\end{equation}
where $X(t;\cdot)$ is the flow map generated by $F[\mu_{t-\tau(t)}]$ in phase space.
\end{Theorem}
\begin{proof}
First we observe that by Lemma \ref{Lipschitz lemma} together with \cite[Theorem 3.10]{CCR} we have local-in-time existence and uniqueness of a measure-valued solution to \eqref{kinetic model} in the sense of \eqref{weak form}. Moreover, this solution exists as long as it is compactly supported in position. Hence, in order to prove the global-in-time existence and uniqueness of solutions to the continuum model \eqref{kinetic model}, we need to estimate the growth of support of $\mu_t.$ So, we set
$$
R_X[\mu_t]:=\max_{x\in \overline{supp \ \mu_t}} |x| \quad \text{for} \quad t\in[0,T].
$$
Moreover, we define
$$
R_X(t):=\max_{-\tau(0)\leq s\leq t} R_X[\mu(s)].
$$
We proceed by steps. Consider $t\in[0,\tau^*]$ and observe then that $t-\tau(t)\in [-\tau(0),0]$. We consider the system of characteristics $X(t;x):[0,\tau^*]\times \RR^d \rightarrow\RR^d$ associated to \eqref{kinetic model}
\begin{equation}
\label{characteristics}
\frac{dX(t;x)}{dt}=F[\mu_{t-\tau(t)}](X(t;x)),
\end{equation}
subject to the initial condition
\begin{equation}
\label{caratteristiche_iniziali}
X(0;x)=x, \quad \text{for} \ x\in \RR^d.
\end{equation}
Then, by Lemma \ref{Lipschitz lemma} there exists a unique solution to \eqref{characteristics}-\eqref{caratteristiche_iniziali} on the time interval $[0,\tau^*]$. Now, by definition of $F[\mu_{t-\tau(t)}]$, choosing either \eqref{F-symmetric} or \eqref{F-nonsymmetric} yields
$$
\frac{d|X(t;x)|}{dt}\leq R_X^{t-\tau(t)}-|X(t;x)|.
$$
Using a continuity argument as in Lemma \ref{Lemma iniziale}, we obtain
$$
R_X^t<R_X^0,
$$
for $t\in[0,\tau^*]$.
Thus, we can construct a unique solution $\mu_t$ to \eqref{characteristics}-\eqref{caratteristiche_iniziali} on the time interval $[0,\tau^*]$, and this solution is compacty supported in the $x$-variable. We can iterate this argument on all the intervals of length $\tau^*$, namely on the intervals of type $[k\tau^*,(k+1)\tau^*]$, with $k=1,2,\ldots$, until we reach the final time $T$. Indeed, note that if $t\in [N\tau^*,(N+1)\tau^*]$, for some $N>0$, then $t-\tau(t)\in [-\tau(0),N\tau^*]$. Moreover, arguing as in \cite{CCR}, we can find \eqref{push-forward mu} and we have that this formulation is equivalent to \eqref{weak form}.
\end{proof}
In order to prove the consensus behavior of the solution to the kinetic model \eqref{kinetic model}, we need a stability estimate.
\begin{Theorem}\label{stability theorem}
Let $\mu^1_t,\mu^2_t\in\mathcal{C}([0,T);\mathcal{M}(\RR^d))$ be two weak solutions to \eqref{kinetic model}, subject to uniformly compactly supported initial data $g^1_s,g^2_s\in \mathcal{C}([-\tau(0),0];\mathcal{M}(\RR^d))$, respectively. Then, there exists a constant $C>0$, depending on $T,$ such that
\begin{equation}
\label{stability}
d_p(\mu^1_t,\mu^2_t)\leq C \max_{s\in [-\tau(0),0]} d_p (g^1_s,g^2_s),
\end{equation}
for all $p \in [1,\infty]$ and $t\in [0,T)$.
\end{Theorem}
\begin{proof}
Let $p \in [1,\infty)$ and we construct again the system of characteristics $X^i(t;x):[0,T]\times \RR^d\rightarrow\RR^d$, for $i=1,2,$
$$
\begin{array}{l}
\displaystyle{ \frac{dX^i(t;x)}{dt}=F[\mu^i_{t-\tau(t)}](X^i(t;x)),}\\
\displaystyle{ X^i(0;x)=x,}
\end{array}
$$
for all $x\in \RR^d$. By Theorem \ref{teorema esistenza}, we know that the measures $\mu^i_t$ have uniformly compact support for $t\in[0,T]$. Hence, the flows $X^i$ are well-defined on this interval. Then, arguing as in \cite{CCR}, we can find that $\mu^i_t=X^i(t;\cdot)\#\mu^i_s$, for any $t,s\in[0,T]$. Moreover, as before, we define
$$
R_{i;X}^T:=\max_{-\tau(0) \leq s\leq t} R_X[\mu^i_s].
$$
We choose an optimal transport map $\mathcal{S}_0(x)$ between $\mu^1_0$ and $\mu^2_0$ with respect to the $p$-Wasserstein distance $d_p$, namely
$$
\begin{array}{l}
\displaystyle{
\mu^2_0=\mathcal{S}_0\#\mu^1_0,}\\
\displaystyle{ d_p(\mu^1_0,\mu^2_0)=\lt(\int_{\RR^d} |x-\mathcal{S}_0(x)|^p\,d\mu^1_0(x)\rt)^{1/p}.}
\end{array}
$$
Furthermore, defining $\mathcal{T}^t:=X^2(t;\cdot)\circ \mathcal{S}_0\circ X^1(t;\cdot)^{-1}$, for $t\in [0,T]$, and using the definition of push-forward, we obtain
\begin{equation}
\label{push}
\mathcal{T}^t\#\mu^1_t=\mu^2_t, \qquad \forall t\in [0,T],
\end{equation}
and
$$
d_p(\mu^1_t,\mu^2_t)\leq \lt(\int_{\RR^d} |x-\mathcal{T}^t(x)|^p\,d\mu^1_t(x)\rt)^{1/p}.
$$
We define
$$u_p(t):=\int_{\RR^d} |x-\mathcal{T}^t(x)|^p\,d\mu^1_t(x), \quad t\in [0, T].$$
Hence, we obtain
\[
d_p(\mu^1_t,\mu^2_t)\leq (u_p(t))^{1/p}, \qquad \forall t\in [0,T].
\]
Moreover, using the fact that $\mathcal{T}^t\circ X^1(t;\cdot)=X^2(t;\cdot)\circ \mathcal{S}_0$, we can rewrite $u_p(t)$ as
$$
u_p(t)=\int_{\RR^d} |X^1(t;x)-X^2(t;\mathcal{S}_0(x))|^p\,d\mu^1_0(x).
$$
We extend the definition of $\mathcal{T}^t$ on the interval $[-\tau(0),0)$ as the optimal transport map between $g^1_t$ and $g^2_t$, and we extend $u_p(t)$ on the same interval, namely
$$
(u_p(t))^{1/p}:=d_p(g^1_t,g^2_t)=\lt(\int_{\RR^d} |x-\mathcal{T}^t(x)|^p\,dg^1_t(x)\rt)^{1/p},
$$
for $t\in [-\tau(0),0]$.
We have that
\begin{align*}
\frac{du_p(t)}{dt}&\leq p\int_{\RR^d} |X^1(t;x)-X^2(t;\mathcal{S}_0(x))|^{p-1}\cr
&\hspace{2cm} \times |F[\mu^1_{t-\tau(t)}](X^1(t;x))-F[\mu^2_{t-\tau(t)}](X^2(t;\mathcal{S}_0(x)))|\,d\mu^1_0(x)\cr
&=:J.
\end{align*}
Using \eqref{push}, we can rewrite $J$ as follows:
$$
J=p\int_{\RR^d} |x-\mathcal{T}^t(x)|^{p-1}|F[\mu^1_{t-\tau(t)}](x)-F[\mu^2_{t-\tau(t)}](\mathcal{T}^t(x))|\,d\mu^1_t(x).
$$
We consider, now, $F$ as in \eqref{F-symmetric}. In this case we have that
\begin{align*}
&|F[\mu^1_{t-\tau(t)}](x)-F[\mu^2_{t-\tau(t)}](\mathcal{T}^t(x))|\cr
&\quad \leq \int_{\RR^d} \bigl| \psi(|x-y|)(y-x)-\psi (|\mathcal{T}^{t-\tau(t)}(y)-\mathcal{T}^t(x)|)(\mathcal{T}^{t-\tau(t)}(y)-\mathcal{T}^t(x)) \bigr| \,d\mu^1_{t-\tau(t)}(y)\cr
&\quad \leq \int_{\RR^d} \bigl| \psi(|x-y|)-\psi(|\mathcal{T}^{t-\tau(t)}(y)-\mathcal{T}^t(x)|)\bigr| \cdot |y-x| \,d\mu^1_{t-\tau(t)}(y)\cr
&\qquad +\int_{\RR^d} |\psi (|\mathcal{T}^{t-\tau(t)}(y)-\mathcal{T}^t(x)|)|\cdot|y-x-(\mathcal{T}^{t-\tau(t)}(y)-\mathcal{T}^t(x))|\,d\mu^1_{t-\tau(t)}(y).
\end{align*}
The first term can be bounded as follows:
\begin{align*}
&\int_{\RR^d} \bigl| \psi(|x-y|)-\psi(|\mathcal{T}^{t-\tau(t)}(y)-\mathcal{T}^t(x)|)\bigr| |y-x| \,d\mu^1_{t-\tau(t)}(y)\cr
&\quad \le L (|x|+R_x^1) \int_{\RR^d} \bigl| \ |y-x|-|\mathcal{T}^{t-\tau(t)}(y)-\mathcal{T}^t(x)|\ \bigr| d\mu^1_{t-\tau(t)}(y)\cr
&\quad \le  L (|x|+R_x^1) \lt(|x-\mathcal{T}^t(x)|+\int_{\RR^d} |y-\mathcal{T}^{t-\tau(t)}(y)|\,d\mu^1_{t-\tau(t)}(y)\rt).
\end{align*}
The second term is bounded by
\begin{align*}
&\int_{\RR^d} |\psi (|\mathcal{T}^{t-\tau(t)}(y)-\mathcal{T}^t(x)|)|\cdot|y-x-(\mathcal{T}^{t-\tau(t)}(y)-\mathcal{T}^t(x))|\,d\mu^1_{t-\tau(t)}(y)\cr
&\quad \leq \int_{\RR^d} (|x-\mathcal{T}^t(x)|+|y-\mathcal{T}^{t-\tau(t)}(y)|)\,d\mu^1_{t-\tau(t)}(y)\cr
&\quad \leq |x-\mathcal{T}^t(x)|+\int_{\RR^d} |y-\mathcal{T}^{t-\tau(t)}(y)|\,d\mu^1_{t-\tau(t)}(y).
\end{align*}
Hence, there exists a constant $C>0$ depending  on the Lipschitz constant $L$  of $\psi$ and on $R_{i,X}^T$ such that
\begin{align*}
J&\leq Cp\int_{\RR^d} |x-\mathcal{T}^t(x)|^p\,d\mu^1_t(x)+Cp\int_{\RR^d} |x-\mathcal{T}^t(x)|^{p-1}\,d\mu^1_t(x)\int_{\RR^d} |y-\mathcal{T}^{t-\tau(t)}(y)|\,d\mu^1_{t-\tau(t)}(y) \cr
&\leq Cpu_p(t)+Cu_p(t-\tau(t)).
\end{align*}
Here we used
\begin{align*}
&\int_{\RR^d} |x-\mathcal{T}^t(x)|^{p-1}\,d\mu^1_t(x)\int_{\RR^d} |y-\mathcal{T}^{t-\tau(t)}(y)|\,d\mu^1_{t-\tau(t)}(y) \cr
&\quad \leq \lt(\int_{\RR^d} |x-\mathcal{T}^t(x)|^p\,d\mu^1_t(x)\rt)^{(p-1)/p}\lt(\int_{\RR^d} |y-\mathcal{T}^{t-\tau(t)}(y)|^p\,d\mu^1_{t-\tau(t)}(y)\rt)^{1/p}\cr
&\quad \leq \frac{p-1}{p}\lt(\int_{\RR^d} |x-\mathcal{T}^t(x)|^p\,d\mu^1_t(x)\rt) + \frac1p \lt(\int_{\RR^d} |y-\mathcal{T}^{t-\tau(t)}(y)|^p\,d\mu^1_{t-\tau(t)}(y)\rt).
\end{align*}
Now, if we take $F$ as in \eqref{F-nonsymmetric}, we have that
\begin{align*}
&|F[\mu^1_{t-\tau(t)}](x)-F[\mu^2_{t-\tau(t)}](\mathcal{T}^t(x))|\cr
&\quad \leq\frac{1}{\psi(R_{1;X}^T)}\lt| \int_{\RR^d} \psi(|x-y|)(y-x)\,d\mu^1_{t-\tau(t)}(y)-\int_{\RR^d}\psi(|\mathcal{T}^t(x)-y|)(y-\mathcal{T}^t(x))\,d\mu^2_{t-\tau(t)}(y)\rt|\cr
&\qquad +\frac{1}{\psi(R_{1;X}^T)\psi(R_{2;X}^T)}\lt| \int_{\RR^d} \psi(|\mathcal{T}^t(x)-y|)(y-\mathcal{T}^t(x))\,d\mu^2_{t-\tau(t)}(y)\rt|\cr
&\hspace{5cm}\times \lt| \int_{\RR^d} \psi (|x-y|)\,d\mu^1_{t-\tau(t)}(y)-\int_{\RR^d} \psi (|\mathcal{T}^t(x)-y|)\,d\mu^2_{t-\tau(t)}(y)\rt|.
\end{align*}
For the first term, we have a similar estimate as before, namely
\begin{align*}
& \lt| \int_{\RR^d} \psi(|x-y|)(y-x)\,d\mu^1_{t-\tau(t)}(y)-\int_{\RR^d}\psi(|\mathcal{T}^t(x)-y|)(y-\mathcal{T}^t(x))\,d\mu^2_{t-\tau(t)}(y)\rt|\cr
&\quad \leq \lt((|x|+R_X^1)||\psi||_{Lip}+1\rt)\lt(|x-\mathcal{T}^t(x)|+\int_{\RR^d} |y-\mathcal{T}^{t-\tau(t)}(y)|\,d\mu^1_{t-\tau(t)}(y)\rt).
\end{align*}
Moreover,
$$
\lt| \int_{\RR^d} \psi(|\mathcal{T}^t(x)-y|)(y-\mathcal{T}^t(x))\,d\mu^2_{t-\tau(t)}(y)\rt|\leq R_{2;X}^T+|\mathcal{T}^t(x)|,
$$
and
\begin{align*}
&\lt| \int_{\RR^d} \psi (|x-y|)\,d\mu^1_{t-\tau(t)}(y)-\int_{\RR^d} \psi (|\mathcal{T}^t(x)-y|)\,d\mu^2_{t-\tau(t)}(y)\rt|\cr
&\quad \leq \int_{\RR^d} |\psi(|x-y|)-\psi(|\mathcal{T}^t(x)-\mathcal{T}^{t-\tau(t))}(y)|)|\,d\mu^1_{t-\tau(t)}(y)\cr
&\quad \leq L \lt(|x-\mathcal{T}^t(x)|+\int_{\RR^d} |y-\mathcal{T}^{t-\tau(t)}(y)|\,d\mu^1_{t-\tau(t)}(y)\rt).
\end{align*}
Hence, as before, there exists a constant $C>0$, which is independent of $p$, such that
$$
\frac{du_p(t)}{dt}\leq Cp(u_p(t)+u_p(t-\tau(t)),
$$
for all $t\in [0,T]$, due to $p \geq 1$. Now, we denote
$$
(\bar{u}_p)^{1/p}:=\max_{s\in[-\tau(0),0]} (u_p(s))^{1/p}=\max_{s\in [-\tau(0),0]}d_p(g^1_s,g^2_s),
$$
and set $w_p(t):=e^{-Cpt}u_p(t)$. Hence, we have that
$$
\frac{dw_p(t)}{dt}\leq Cpw_p(t-\tau(t)).
$$
Consider $t\in [0,\tau^*]$. Since $t-\tau(t)\in [-\tau(0),0]$, then by comparison principle we have
$$
w_p(t)\leq \bar{u}_p (1+Cp\tau^*).
$$
Inductively, we can prove that for any $t\in ((k-1)\tau^*,k\tau^*]$, with $k=1,2,\ldots$, until we reach $T$, we obtain
$$
w_p(t)\leq \bar{u}_p(1+Cp\tau^*)^k,
$$
i.e.
$$
u_p(t)\leq \bar{u}_pe^{Cpt}(1+Cp\tau^*)^k \leq \bar{u}_pe^{CpT}(1+Cp\tau^*)^K,
$$
where $K$ is the natural number such that $T\in ((K-1)\tau^*, K\tau^*].$
Hence, we obtain \eqref{stability}, just recalling that
\[
d_p(\mu^1_t,\mu^2_t)\leq (u_p(t))^{1/p} \leq (\bar{u}_p)^{1/p}e^{CT}(1+Cp\tau^*)^{\frac K p}
\]
for any $t\in [0,T]$ and $p \in [1,\infty)$. Furthermore, since
\[
(\bar{u}_p)^{1/p} \to \max_{s\in [-\tau(0),0]}d_\infty(g^1_s,g^2_s) \quad \mbox{and} \quad
(1+Cp\tau^*)^{\frac K p}\rightarrow 1 \mbox{ as } p \to \infty,
\]
we also have
\[
d_\infty(\mu^1_t,\mu^2_t) \leq C\max_{s\in [-\tau(0),0]}d_\infty(g^1_s,g^2_s),
\]
for any $t\in [0,T]$.
\end{proof}
This stability result is useful in order to have a rigorous passage from the particle model \eqref{modello}-\eqref{modello,iniziale} to the continuum equation \eqref{kinetic model}. Indeed, fix $g(s)\in \mathcal{C}([-\tau(0),0];\mathcal{M}(\RR^d))$, with compact support, namely $supp \ g_s \subset B^d(0,R)$ for some $R>0$ and for all $s\in [-\tau(0),0]$. Consider a family of $N$-particle approximations of $g(s)$, $\{g^N_s\}_{N\in \nat}$, i.e.
$$
g^N_s:=\frac1N\sum_{i=1}^N \delta(x-x_{i,0}(s))
$$
for $s\in[-\tau(0),0]$, where $x_{i,0}\in \mathcal{C}([-\tau(0),0];\RR^d)$ are chosen such that
\begin{equation}\label{stima distanza iniziale}
\max_{s\in [-\tau(0),0]} d_p(g_s,g^N_s)\rightarrow 0 \qquad \text{as} \qquad N\rightarrow +\infty.
\end{equation}
Now let $\{ x_i^N \}_{N\in \nat}$ be the solution to the discrete model \eqref{modello}, with inital conditions given by
$$
x_i(s)=x_i^0(s),\qquad \forall s\in [-\tau(0),0].
$$
Moreover, let
\begin{equation}\label{ftN}
\mu^N_t :=\frac1N\sum_{i=1}^N  \delta (x-x_i^N(t)), \qquad \forall t\in [0,T].
\end{equation}
Then, we have that $\mu^N_t$ is a measure-valued solution to the kinetic model \eqref{kinetic model}, in the sense of \eqref{weak form}. Moreover, if $\mu_t\in \mathcal{C}([0,T);\mathcal{M}(\RR^d))$ is a weak solution to \eqref{weak form} with initial datum $g_s$, then according to Theorem \ref{stability theorem} there exists a constant $C>0$, depending only on $\psi$, $R$ and $T$, such that
$$
d_p(\mu_t,\mu^N_t)\leq C \max_{s\in [-\tau(0),0]} d_p(g_s,g^N_s),
$$
for all $t\in[0,T)$. This means that $\mu^N_t$ is an approximation of $\mu_t$, namely $\mu^N_t\rightarrow \mu_t$, uniformly on $[0,T)$, as $N\rightarrow +\infty$. This gives us a convergence result of the solution of \eqref{kinetic model} to consensus. Indeed, we define the position diameter for compactly supported measure $g\in \mathcal{M}(\RR^d)$ as follows:
$$
d_X[g]:=\text{diam} (supp \ g).
$$
Hence, we have the following theorem.
\begin{Theorem}
Let $\mu_t \in \mathcal{C}([0,T);\mathcal{M}(\RR^d))$ be a solution to \eqref{kinetic model}, in the sense of \eqref{weak form} on the time interval $[0,T)$ with compactly supported initial datum $g_s\in \mathcal{C}([-\tau(0),0];\mathcal{M}(\RR^d))$. Moreover consider $F$ as in \eqref{F-symmetric} or \eqref{F-nonsymmetric}. If
\begin{equation}
\label{stima tau barrato}
\overline \tau <\ln \left (
1+ \frac {\psi^3 (2R)} {2 + \psi^2 (2R)} \, \frac{1-c} \lambda
\right ) ,
\end{equation}
then, $\mu_t$ satisfies
\begin{equation}
\label{flocking_kinetic}
d_X[\mu_t]\leq \left( \max_{s\in[-\tau(0),0]} d_X[g_s]\right) e^{-Ct},
\end{equation}
where $C$ is a positive constant independent of $t$.
\end{Theorem}
\begin{proof}
As before, we construct the family of $N$-particle approximations of $g_s$, $\{ g^N_s\}_{N\in\nat}$, i.e.
$$
g^N_s:=\frac1N\sum_{i=1}^N \delta (x-x_i^0(s))
$$
for $s\in [-\tau(0),0]$, where $x_i^0\in \mathcal{C}([-\tau(0),0];\RR^d)$ satisfy \eqref{stima distanza iniziale}. Let $\{ x_i^N\}$ be the solution to \eqref{modello}, subject to the initial data $x_i(s)=x_i^0(s)$, for $s\in[-\tau(0),0]$. Then, since \eqref{stima tau barrato} holds, from Theorem \ref{flocking particle} there exists $C_1>0$ independent of $t$ and $N$ such that
$$
d_X(t)\leq \left( \max_{s\in [-\tau(0),0]} d_X(s) \right) e^{-C_1 t}
$$
for $t\in[0,T)$, where $d_X(t)$ is the diameter  defined in \eqref{diameter}. Moreover, let $\mu^N_t$ be as in \eqref{ftN}. We know that this is a solution to \eqref{kinetic model} in the sense of \eqref{weak form}. Now, if we fix $T>0$, then by Theorem \ref{stability theorem} there exists a constant $C_2>0$ independent of $N$ such that
$$
d_p(\mu_t,\mu^N_t)\leq C_2 \max_{s\in[-\tau(0),0]} d_p(g_s,g^N_s)
$$
for $t\in[0,T)$. Now, letting $N\rightarrow+\infty$ yields $d_X[\mu_t]=d_X(t)$ for all $t\in[0,T)$ and $d_X[g_s]=d_X(s)$ for all $s\in [-\tau(0),0]$. This implies that
$$
d_X[\mu_t]\leq \left( \max_{s\in[-\tau(0),0]} d_X[g_s]\right) e^{-C_1t}
$$
for $t\in[0,T)$. Since $T$ can be chosen arbitrarily, we obtain \eqref{flocking_kinetic}.
\end{proof}

\section{Some numerical tests}\label{numerics}
In this section, we present several numerical experiments for the particle system \eqref{modello} with \eqref{symmetric} or \eqref{notsymmetric} showing the asymptotic time behavior of solutions. For that, we use the built-in {\it dde23} Matlab command, which solves delay differential equations with constant delays. We consider the one dimensional case and take the communication weight function $\psi$ either in \eqref{symmetric} or \eqref{notsymmetric} as 
\begin{equation}\label{eq_psi}
\psi(r) = \frac{1}{(1+r^2)^{\beta}}, \quad r, \ \beta >0,
\end{equation}
the coupling strength $\lambda = 1$, the number of particles $N = 10$, and the initial data are 
\begin{align*}
&x_1(t) = -3, \quad x_2(t) = 7, \quad x_3(t) = 5, \quad x_4(t) = -6, \quad x_5(t) = -1, \cr
&x_6(t) = -8, \quad x_7(t) = -4, \quad x_8(t) = -5, \quad x_9(t) = 10, \quad \mbox{and} \quad x_{10}(t) = 1, \quad \mbox{for} \quad t \leq 0
\end{align*}
that are integers drawn from the discrete uniform distribution on the interval $[-10, 10]$.

\subsection{The particle system \eqref{modello} with \eqref{symmetric}} 
In this part, we consider the particle \eqref{modello} with \eqref{symmetric}. In Figure \ref{sym_long}, we show the time evolution of solutions $\{x_i\}_{i=1}^{10}$ with $\beta = 1$ in \eqref{eq_psi} and different values of time delays, $\tau = 1, 5, 10$, and $50$. For the time delay $\tau = 1$, we cannot see the oscillatory behavior of solutions, however, this behavior appears for $\tau=5, 10$, and $50$. Furthermore, as strength of time delay increases, we need more time to have the consensus behavior and the oscillatory behavior is better observed.

\begin{figure}[t]
\centering
\includegraphics[scale=0.38]{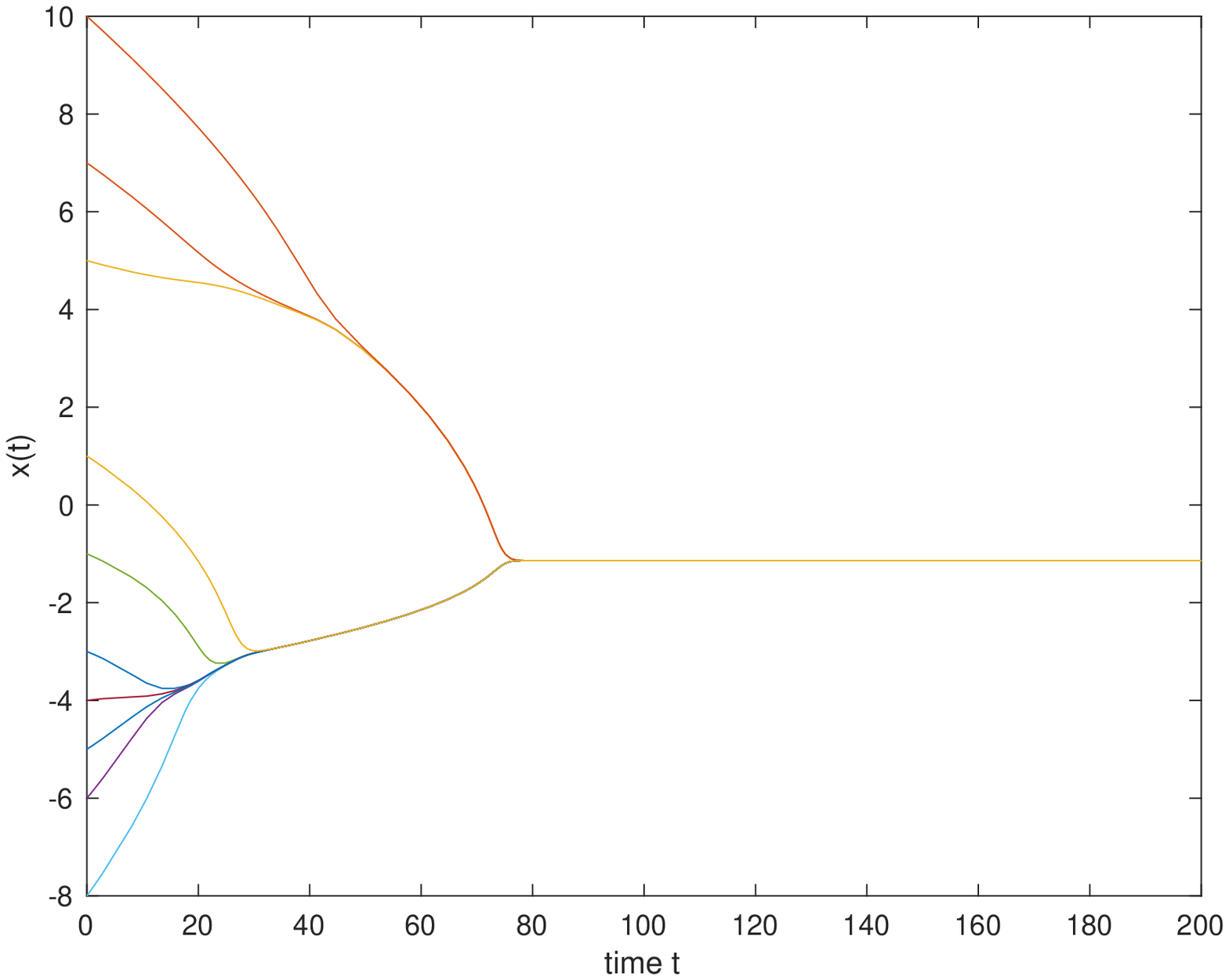}
\includegraphics[scale=0.38]{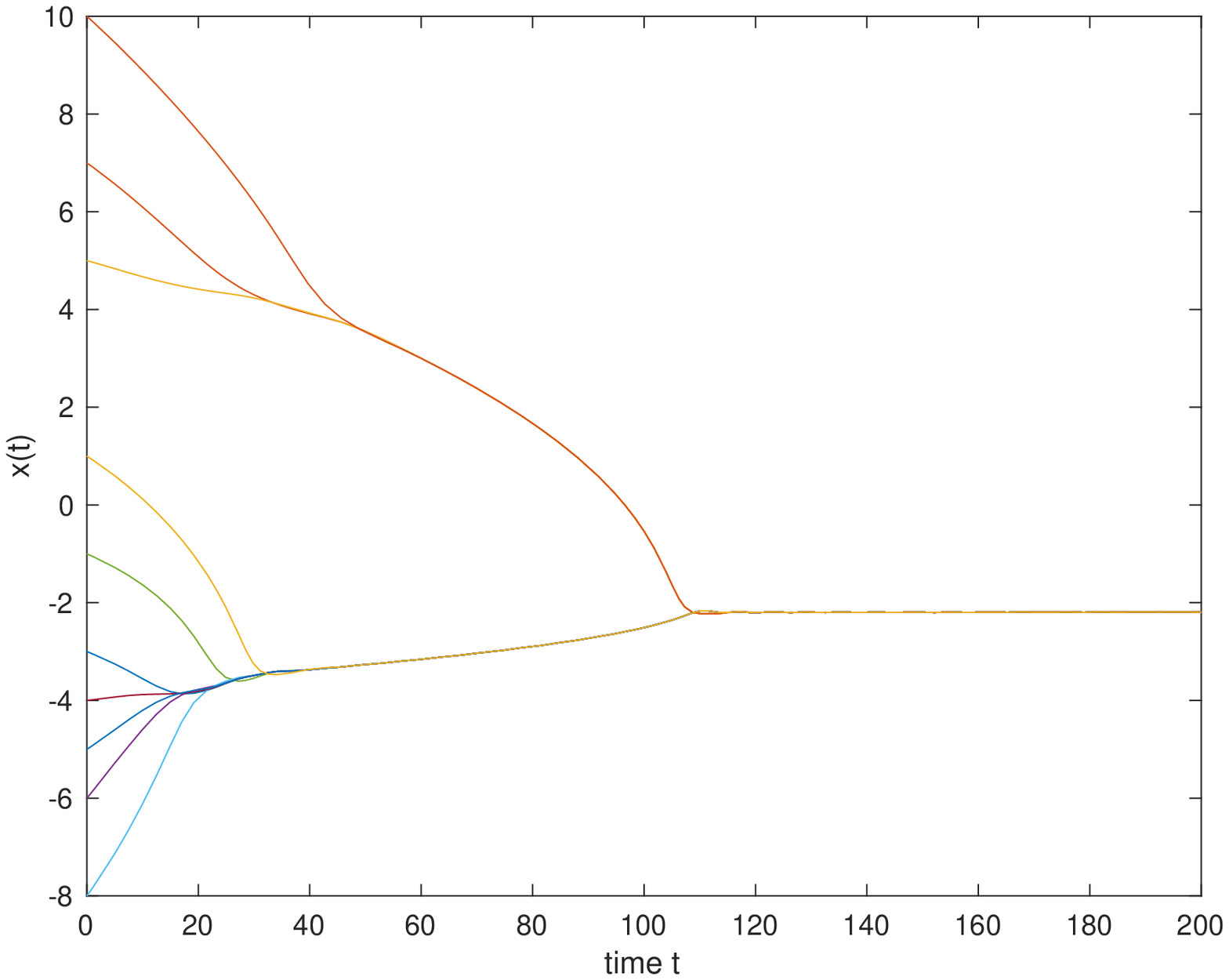}
\includegraphics[scale=0.38]{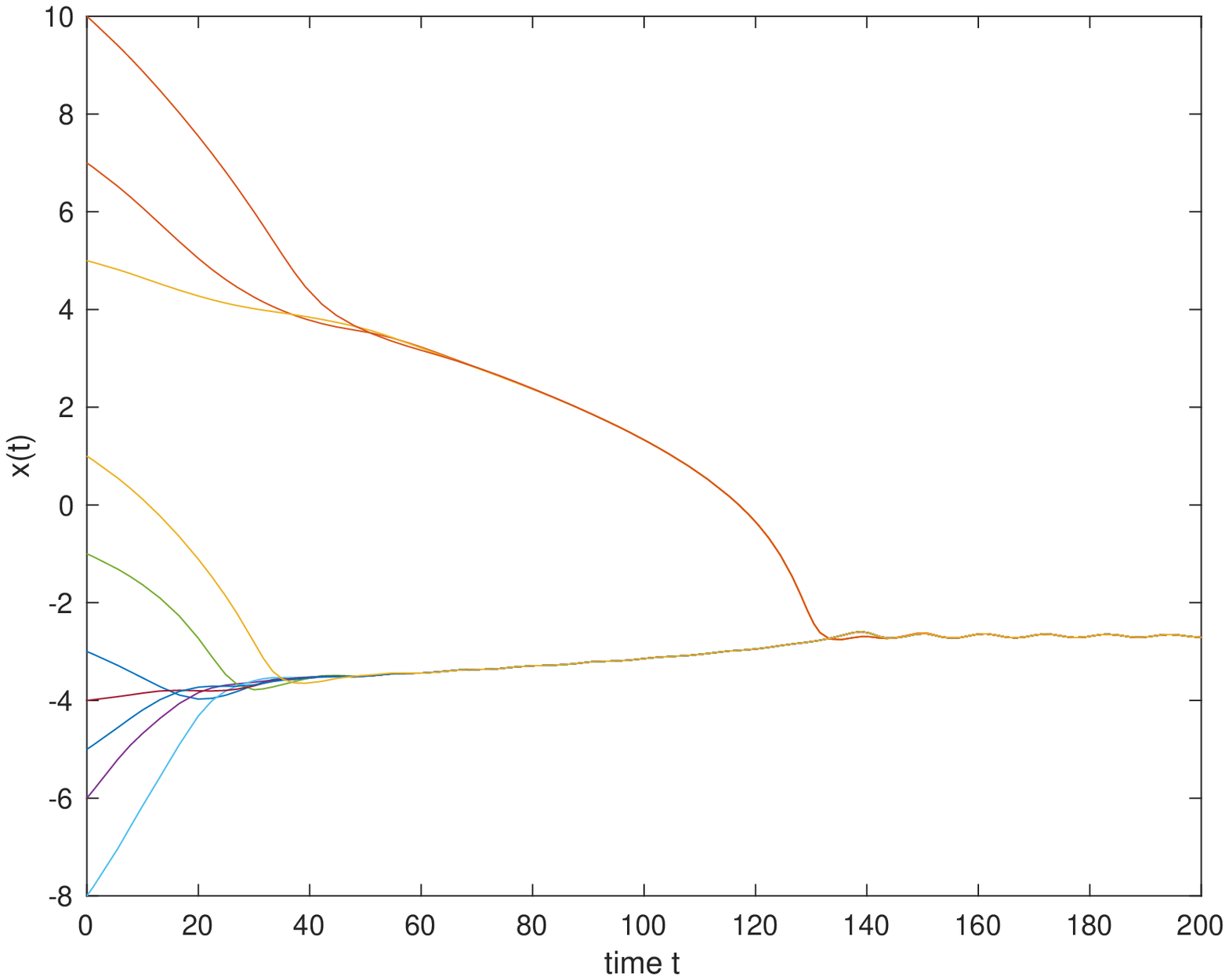}
\includegraphics[scale=0.38]{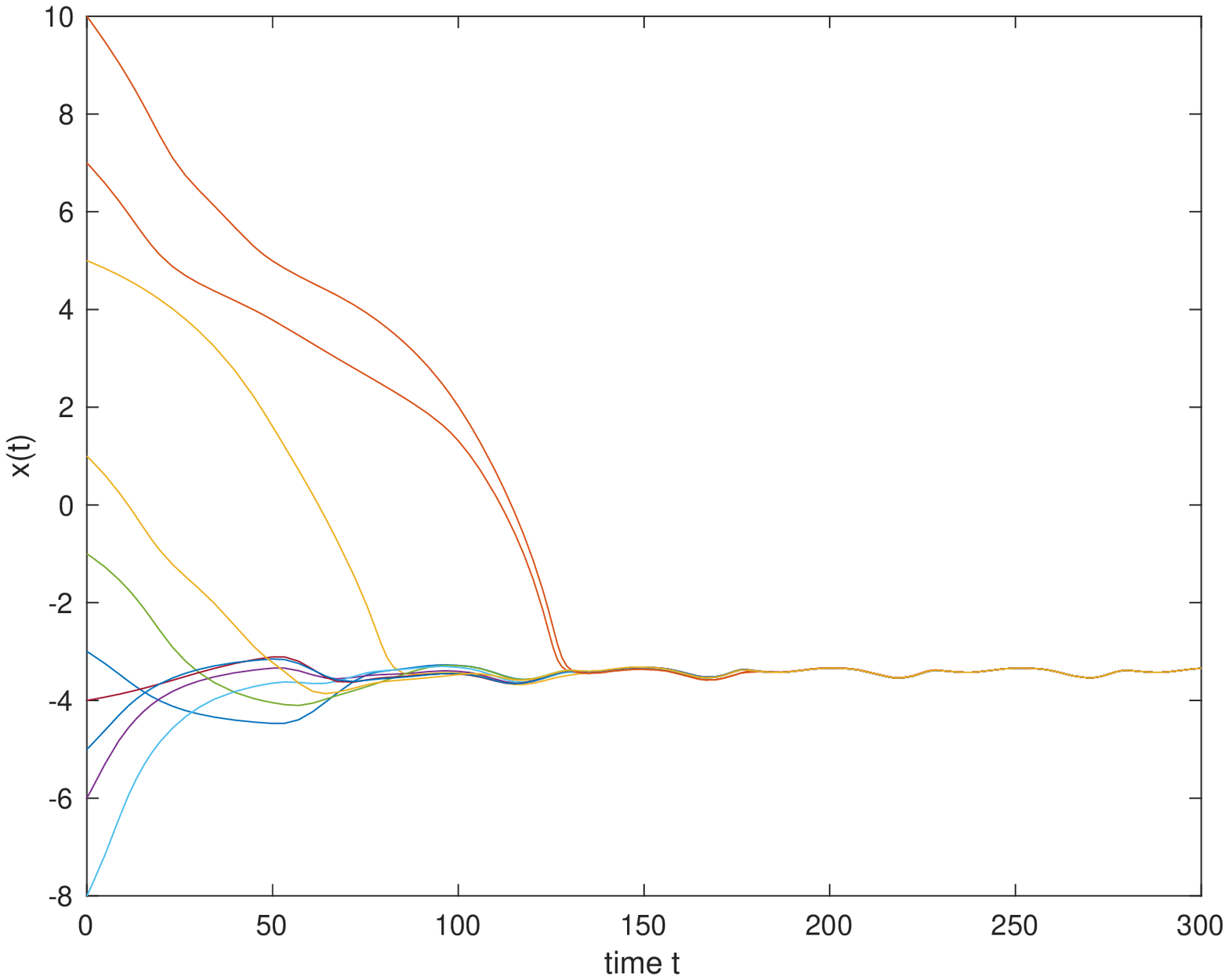}
\caption{Communication rates \eqref{symmetric} \& $\beta = 1$: time evolution of solutions with different strengths of time delays;  $\tau=1$ (top left), $\tau=5$ (top right), $\tau=10$ (bottom left), $\tau=50$ (bottom right).
}
\label{sym_long}
\end{figure}

We next take into account short-range interactions compared to the previous case; we chose $\beta = 3$ in the weight function $\psi$ in \eqref{eq_psi}. In this case, it shows the two clusters formation of solutions as time goes on, not fully consensus behavior, see Figure \ref{sym_short}. Note that multi-cluster formation of solutions to the particle system \eqref{modello} with a compactly supported weight function is investigated in \cite{JM}. We also provide the time evolution of solutions on the time interval $[0,200]$ or $[0,300]$ in the zoomed images in Figure \ref{sym_short} to take a better look at the oscillatory behavior of solutions depending on the strengths of time delays.

\begin{figure}[ht]
\centering
\pgfmathsetlength{\imagewidth}{\linewidth}%
\pgfmathsetlength{\imagescale}{\imagewidth/524}%
\begin{tikzpicture}[x=\imagescale,y=-\imagescale]
\node[anchor=south west] at (0,0) {\includegraphics[scale=0.38]{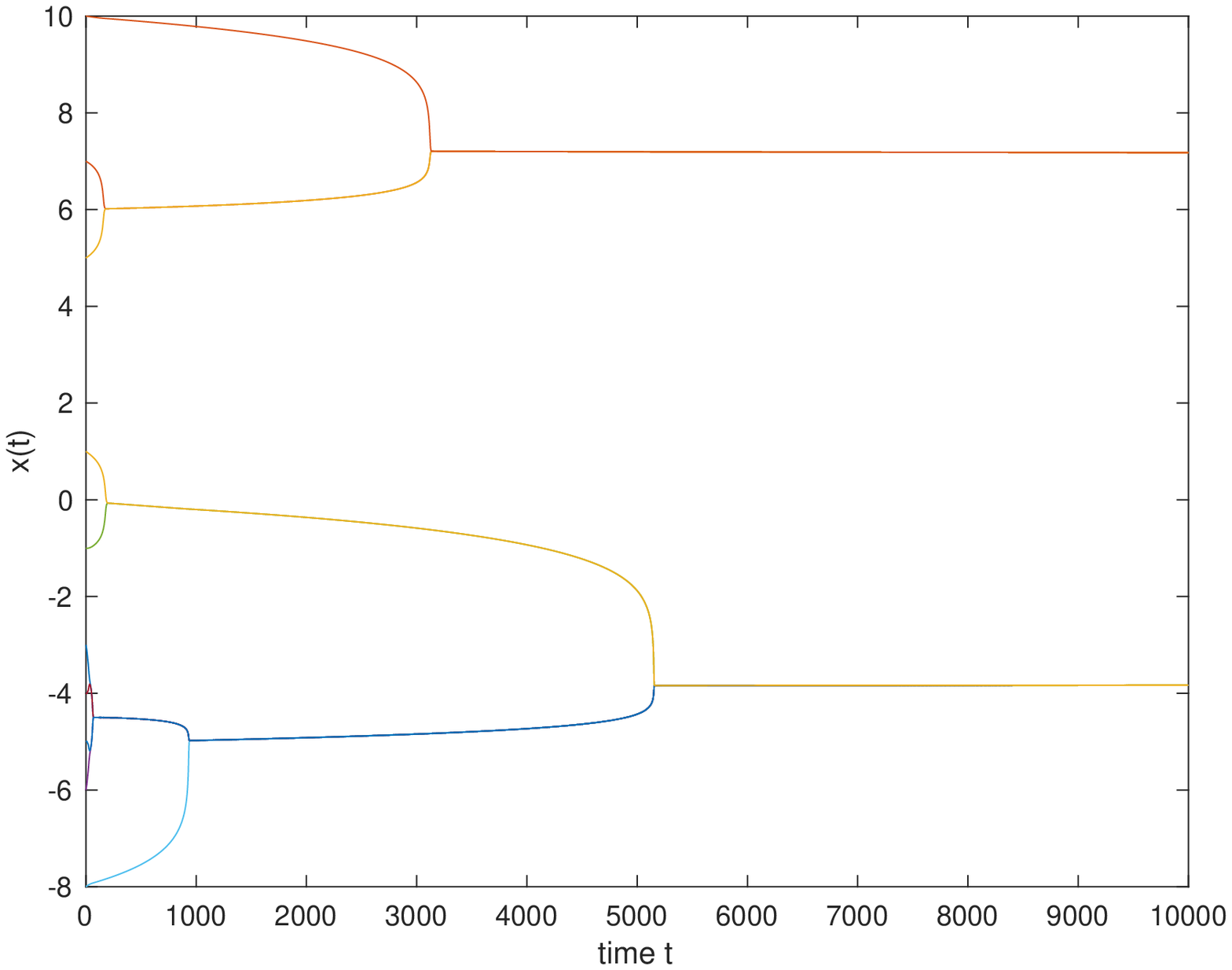}};
\node[anchor=south west] at (170,-70) {\includegraphics[scale=0.15]{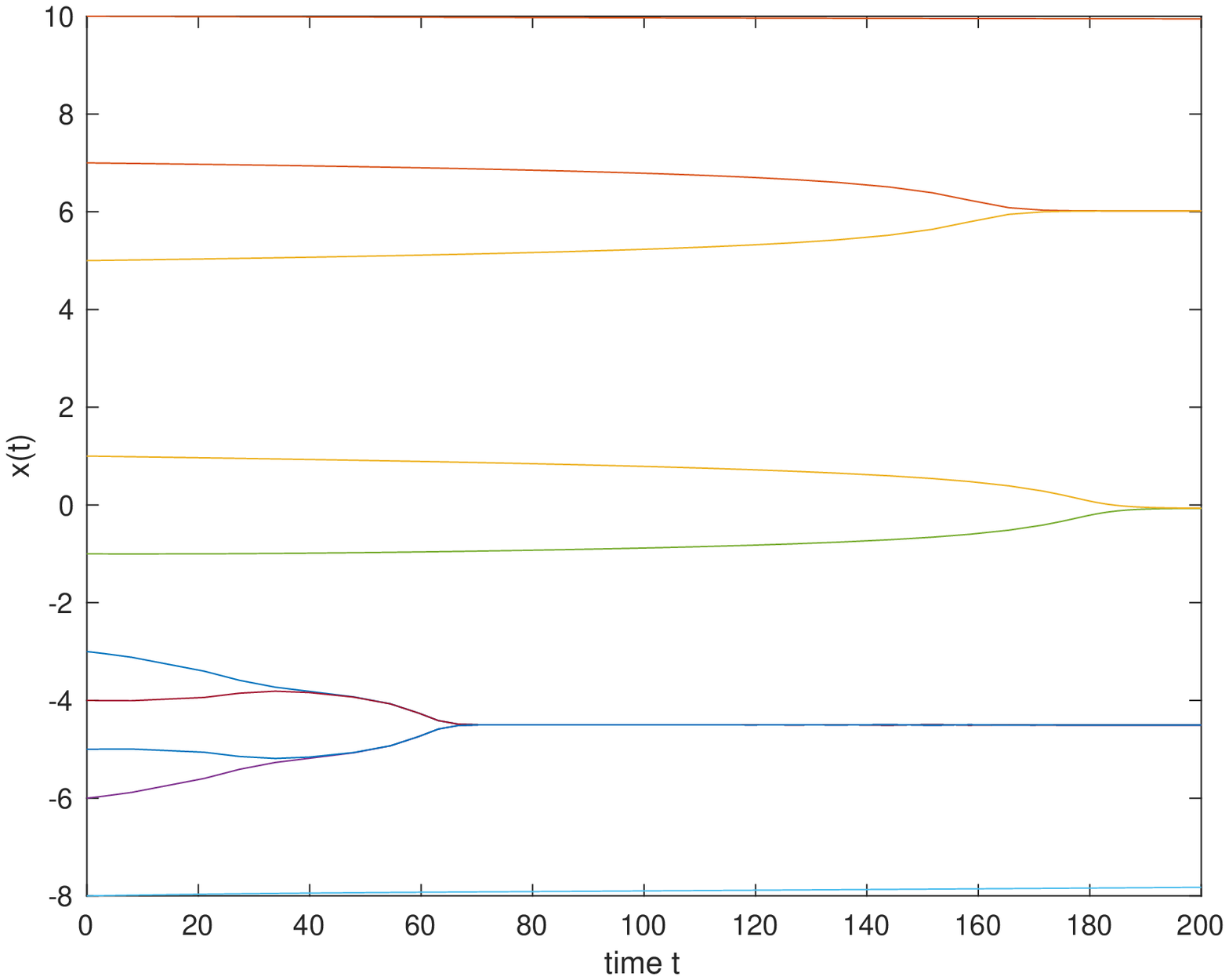}};
\node[anchor=south west] at (260,0) {\includegraphics[scale=0.38]{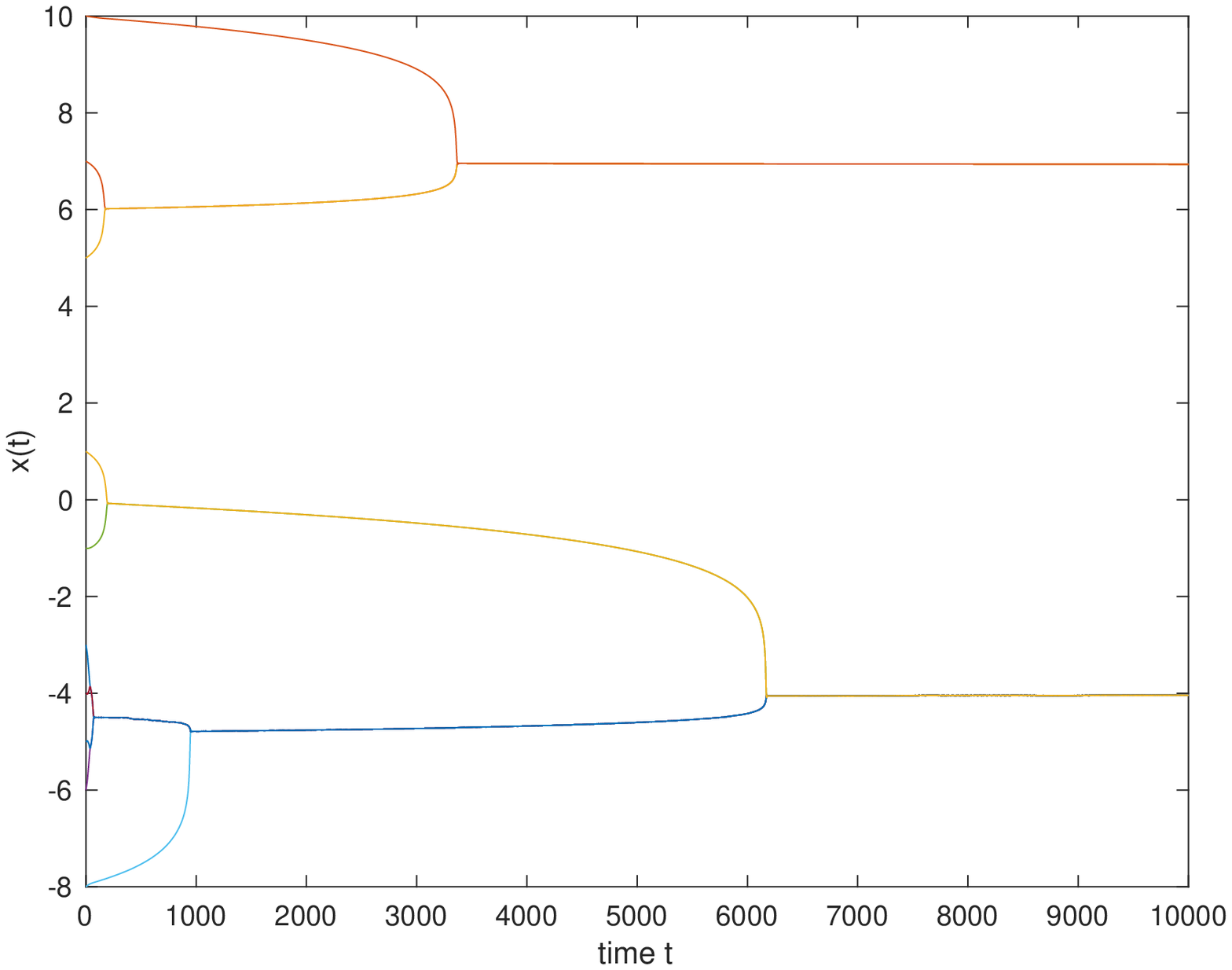}};
\node[anchor=south west] at (432,-70) {\includegraphics[scale=0.15]{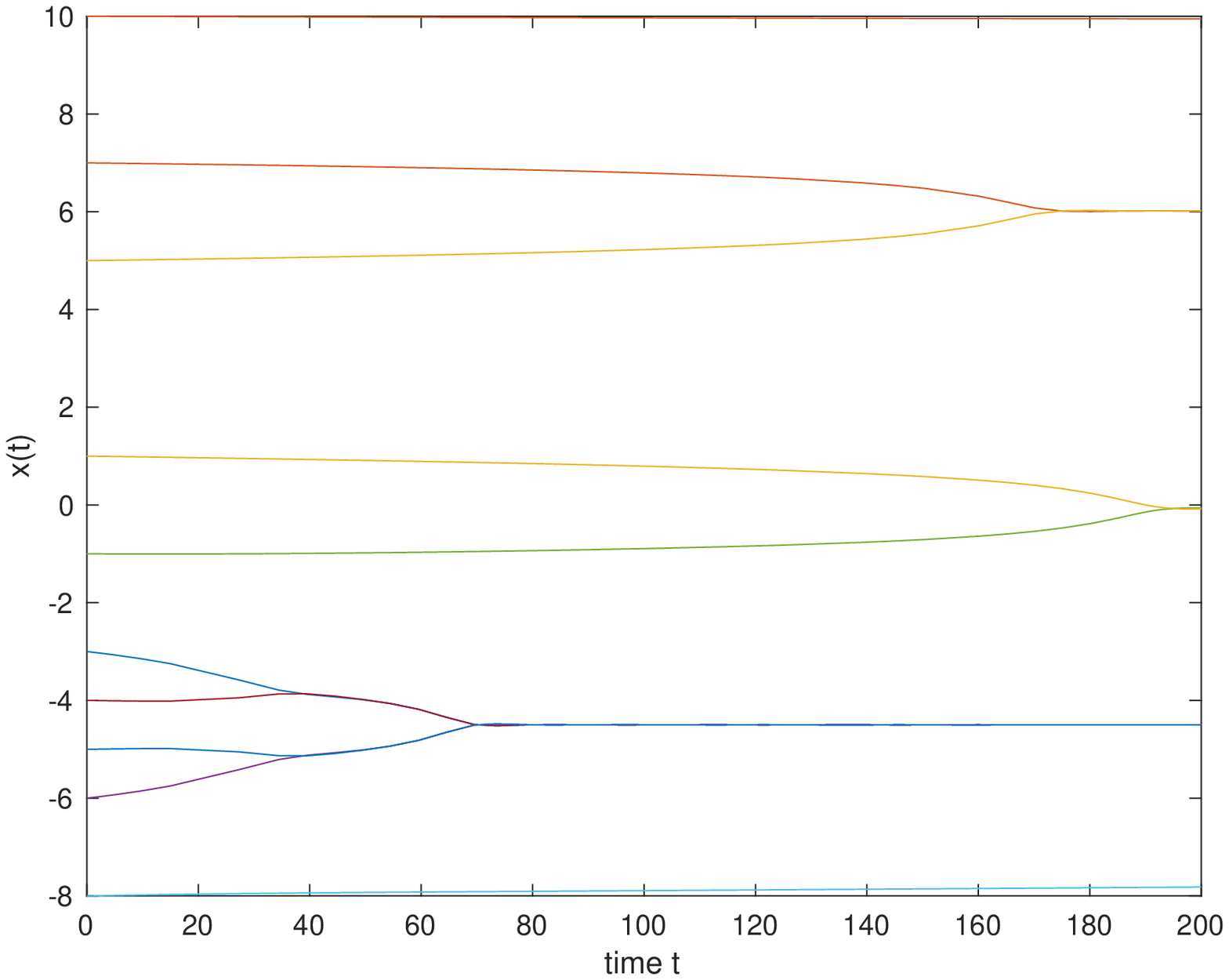}};
\node[anchor=north west] at (0,0) {\includegraphics[scale=0.38]{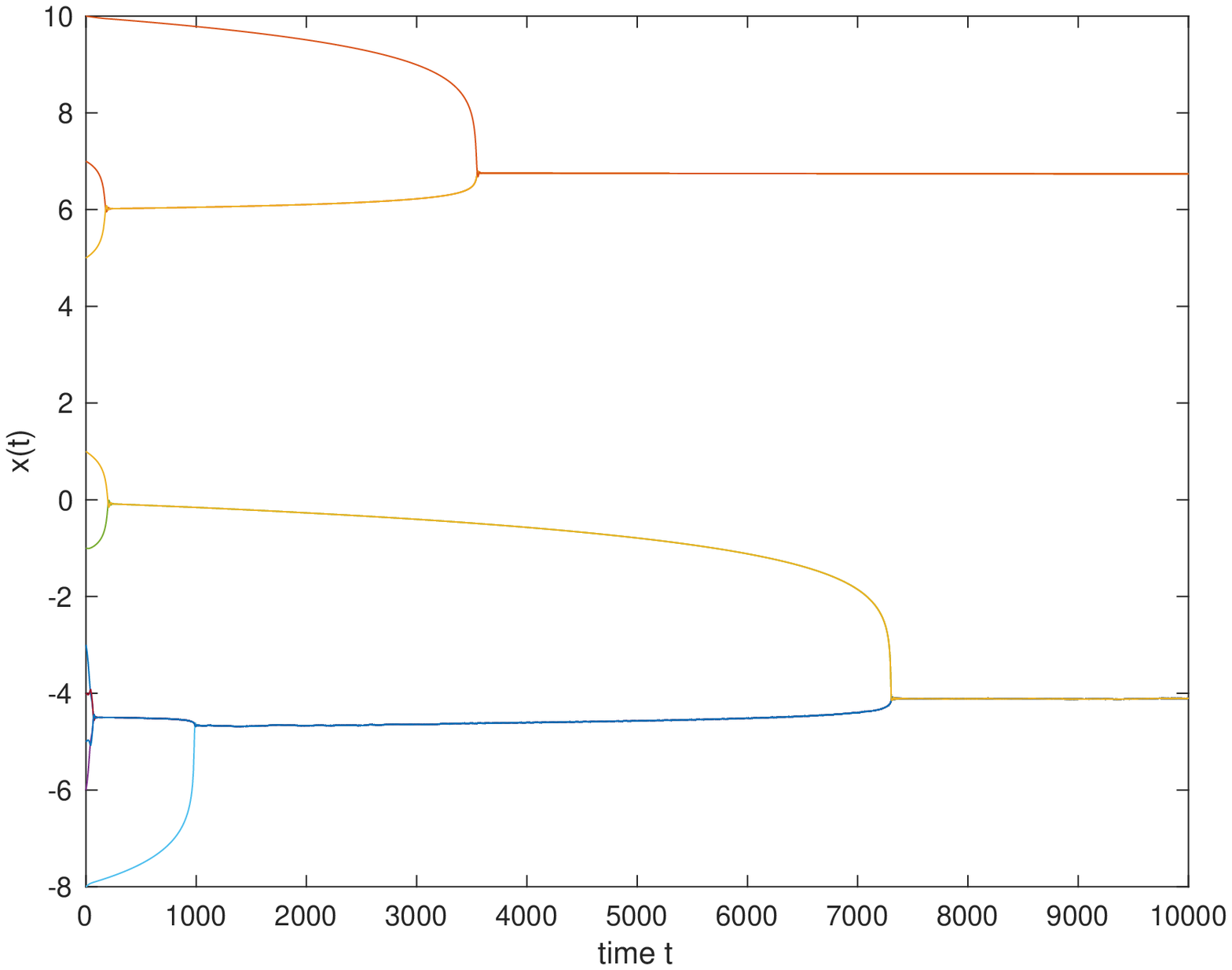}};
\node[anchor=north west] at (170,45) {\includegraphics[scale=0.15]{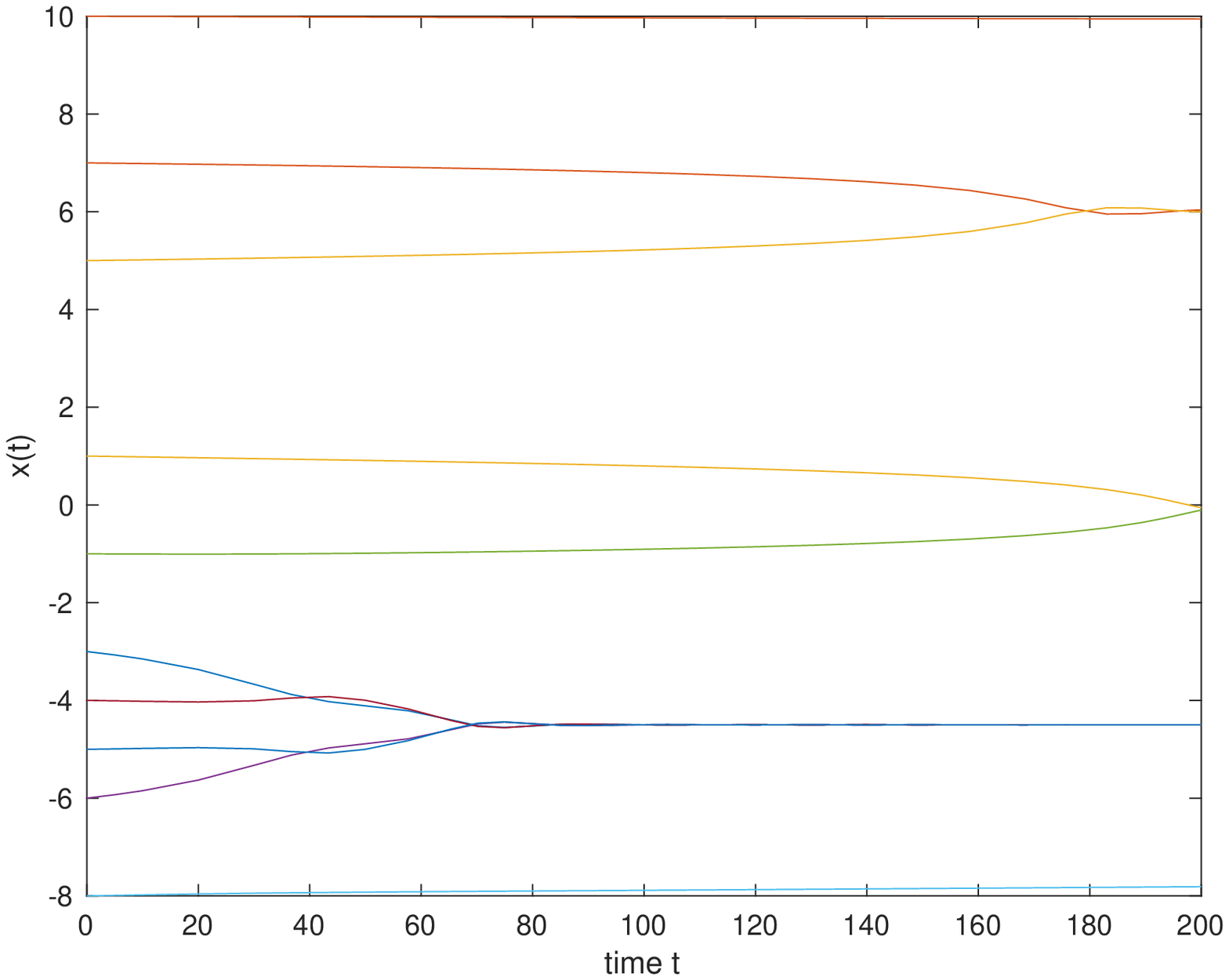}};
\node[anchor=north west] at (260,0) {\includegraphics[scale=0.38]{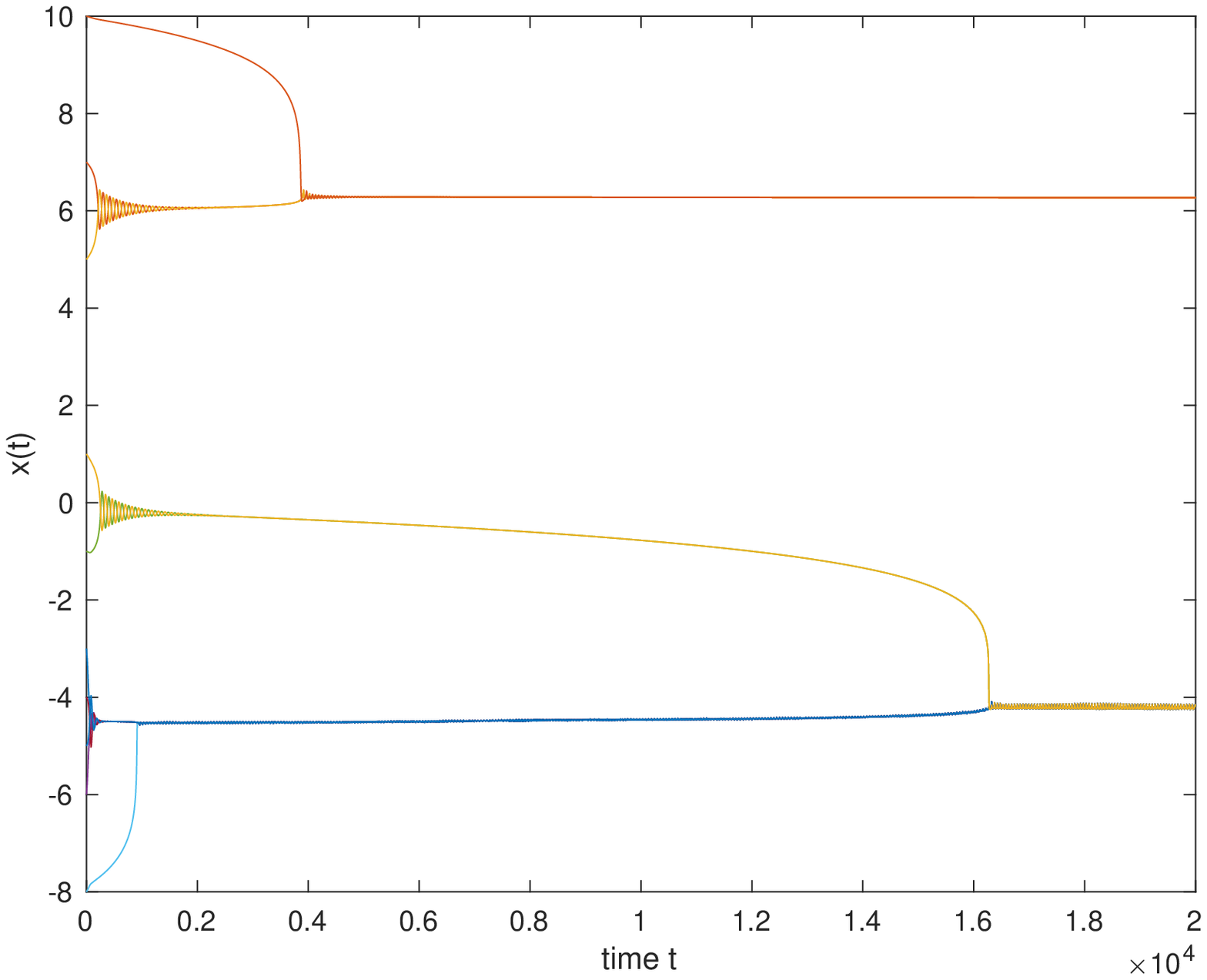}};
\node[anchor=north west] at (432,45) {\includegraphics[scale=0.15]{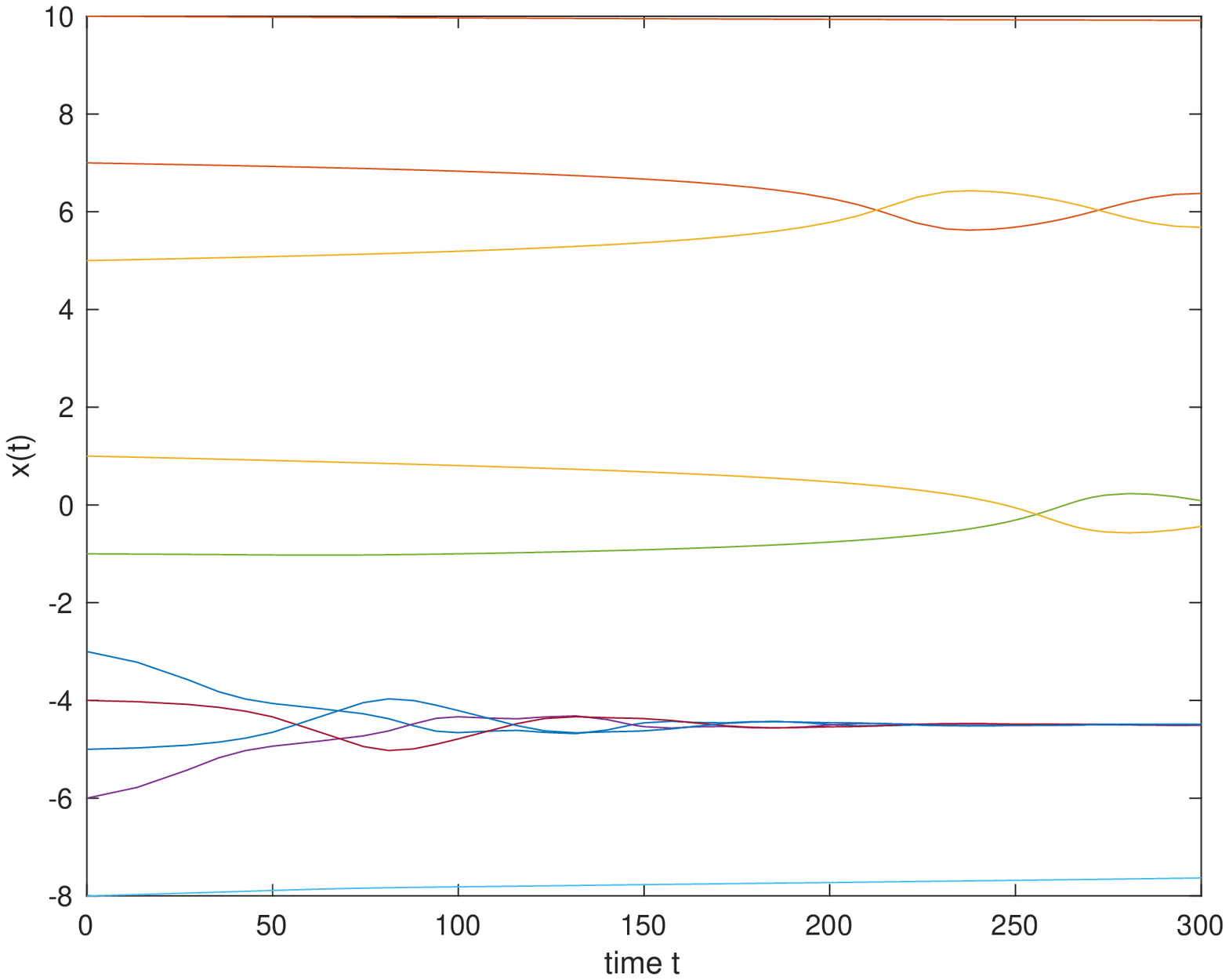}};
\end{tikzpicture}
\caption{Communication rates \eqref{symmetric} \& $\beta = 3$: time evolution of solutions with different strengths of time delays;  $\tau=1$ (top left), $\tau=5$ (top right), $\tau=10$ (bottom left), $\tau=50$ (bottom right).
}
\label{sym_short}
\end{figure}

\subsection{The particle system \eqref{modello} with \eqref{notsymmetric}} 
In this subsection, we consider the particle system \eqref{modello} with \eqref{notsymmetric}. Similarly as before, we first investigate the time evolution of solutions for $\beta=1$ in Figure \ref{notsym_long}. As expected, the consensus behavior of solutions is achieved faster in this case than in the previous case, see also \cite[Section 2]{CCP} for the comparison between the Cucker-Smale flocking model and the Cucker-Smale flocking model with a normalized weight. Compared to the previous case, see Figure \ref{sym_long}, it seems that the particle system \eqref{modello} with \eqref{notsymmetric} is more sensitive to the strength of time delay; multi-cluster formation is not observed during the time evolution for $\tau = 5, 10, 50$, and after the consensus is achieved, it still highly oscillates, see the case with $\tau = 50$ in Figure \ref{notsym_long}.

\begin{figure}[t]
\centering
\includegraphics[scale=0.38]{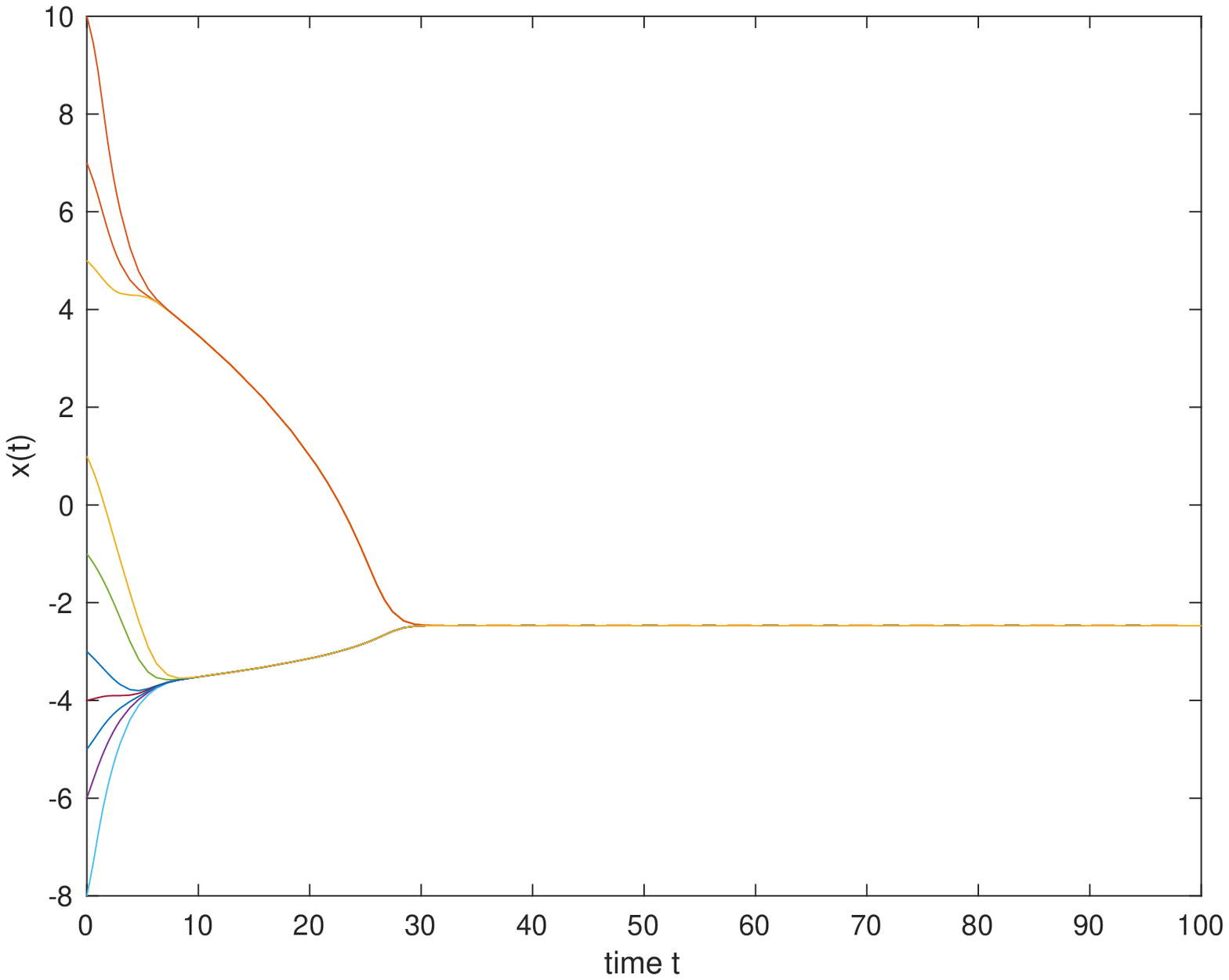}
\includegraphics[scale=0.38]{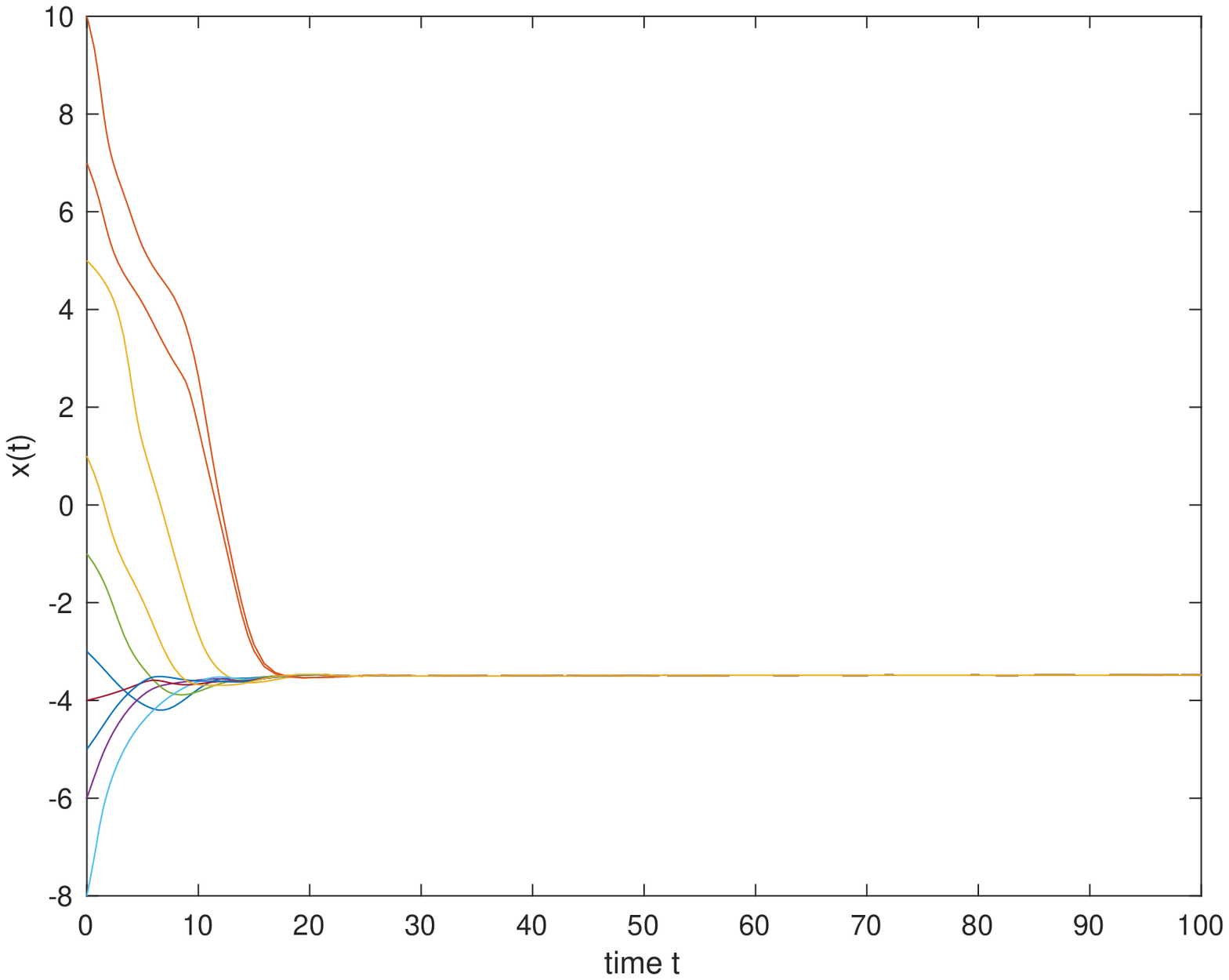}
\includegraphics[scale=0.38]{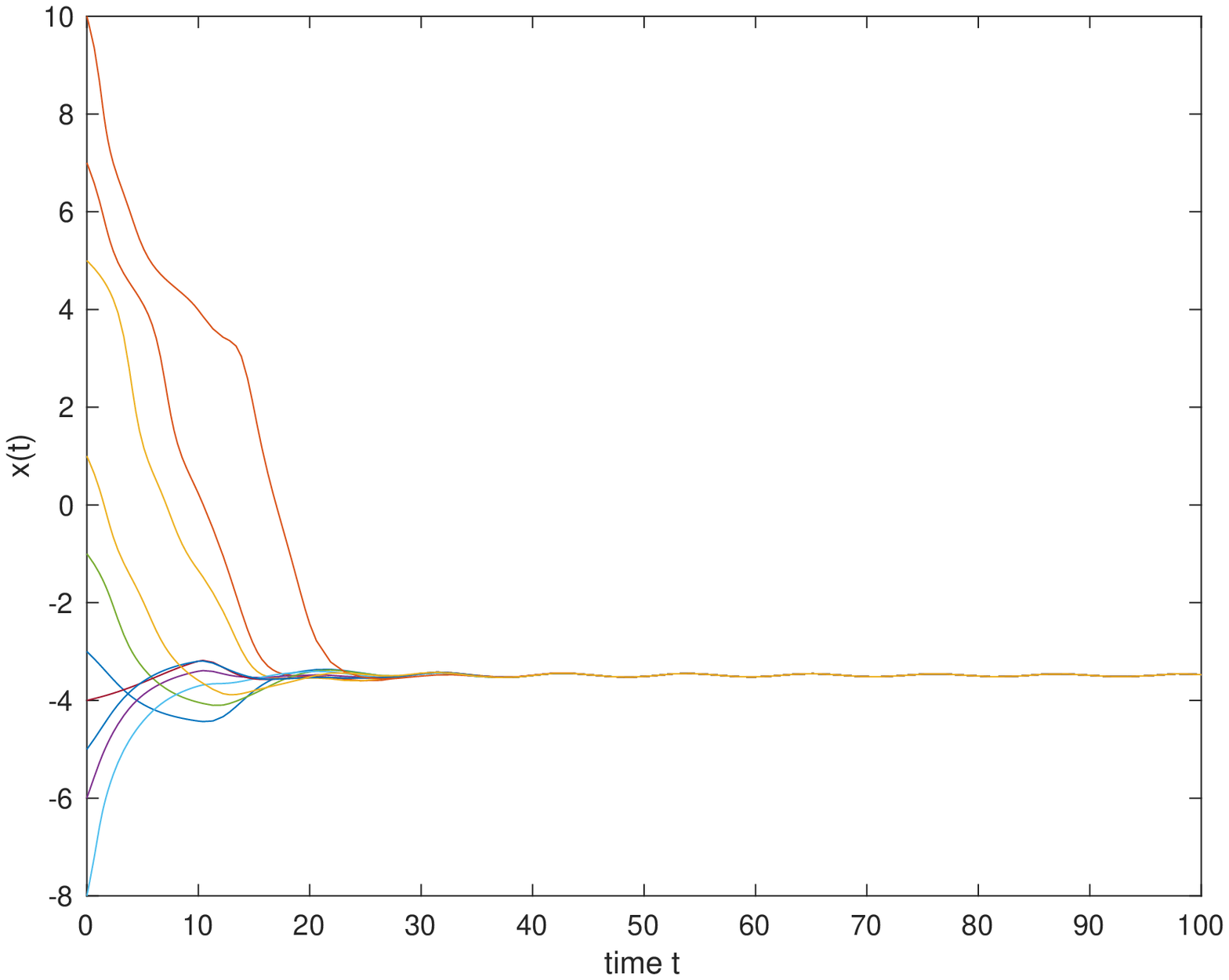}
\includegraphics[scale=0.38]{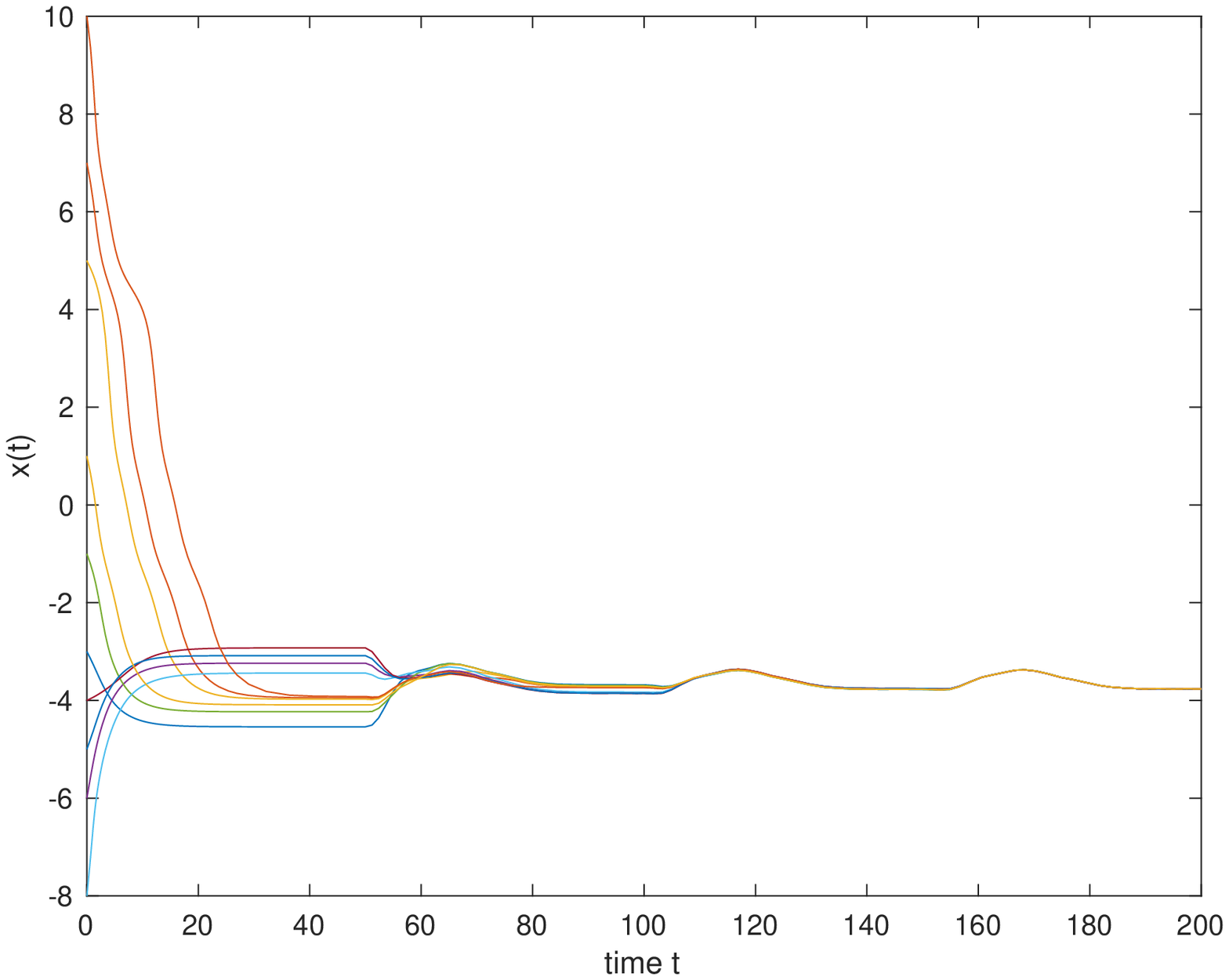}
\caption{Communication rates \eqref{notsymmetric} \& $\beta = 1$: time evolution of solutions with different strengths of time delays;  $\tau=1$ (top left), $\tau=5$ (top right), $\tau=10$ (bottom left), $\tau=50$ (bottom right).
}
\label{notsym_long}
\end{figure}

We finally provide the time evolution of solutions for the case $\beta=3$ in Figure \ref{notsym_short}. Again, in this case, we have the two-cluster formation of solutions. Similarly as before, we put the time evolution of solutions on the time interval $[0,100]$ or $[0,200]$ in the zoomed images in Figure \ref{notsym_short} to have a closer look at the oscillatory behavior of solutions with different values of time delays. As mentioned before, we observe the highly oscillatory behavior of solutions as the strength of time delay increases.

\begin{figure}[ht]
\centering
\pgfmathsetlength{\imagewidth}{\linewidth}%
\pgfmathsetlength{\imagescale}{\imagewidth/524}%
\begin{tikzpicture}[x=\imagescale,y=-\imagescale]
\node[anchor=south west] at (0,0) {\includegraphics[scale=0.38]{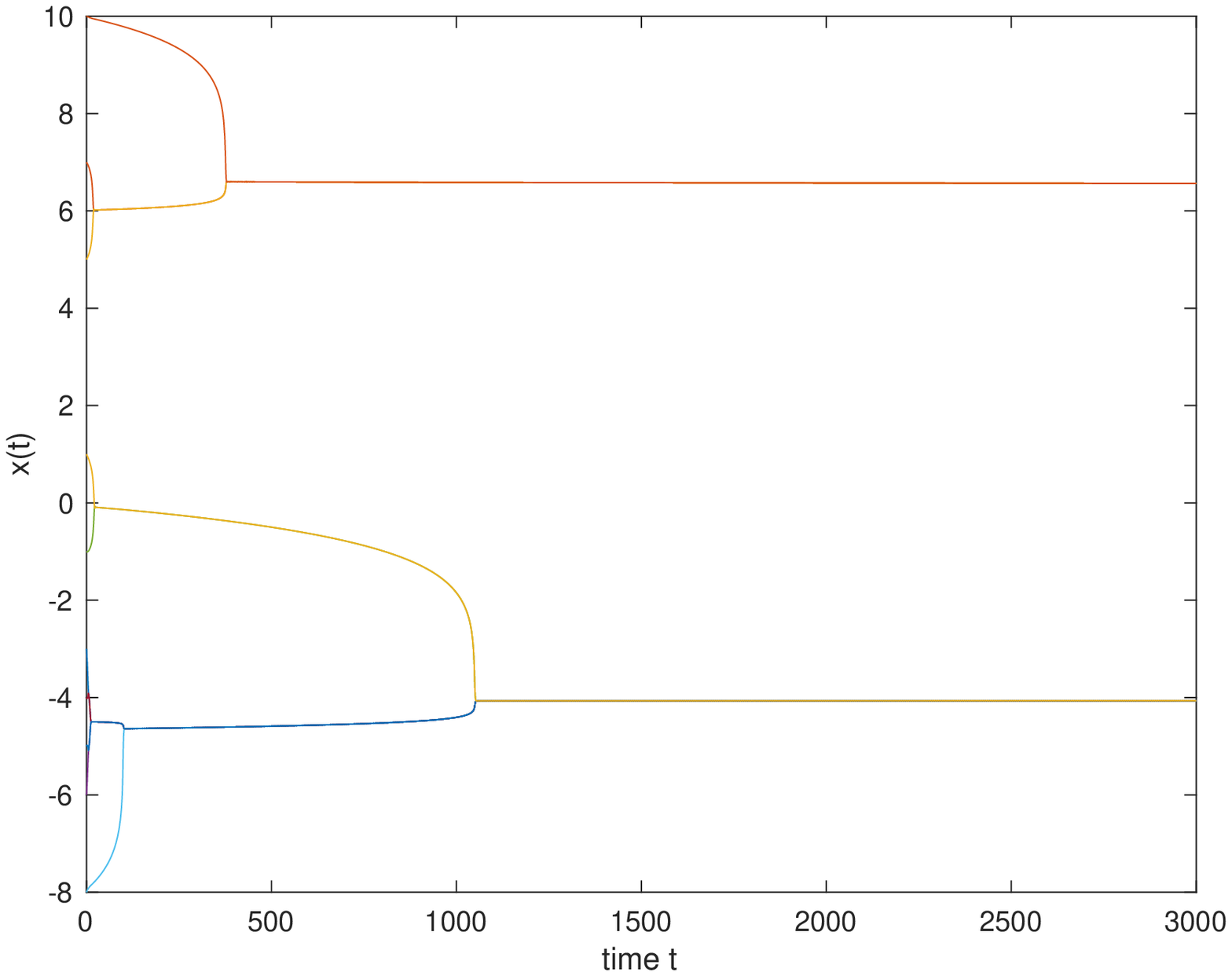}};
\node[anchor=south west] at (170,-65) {\includegraphics[scale=0.15]{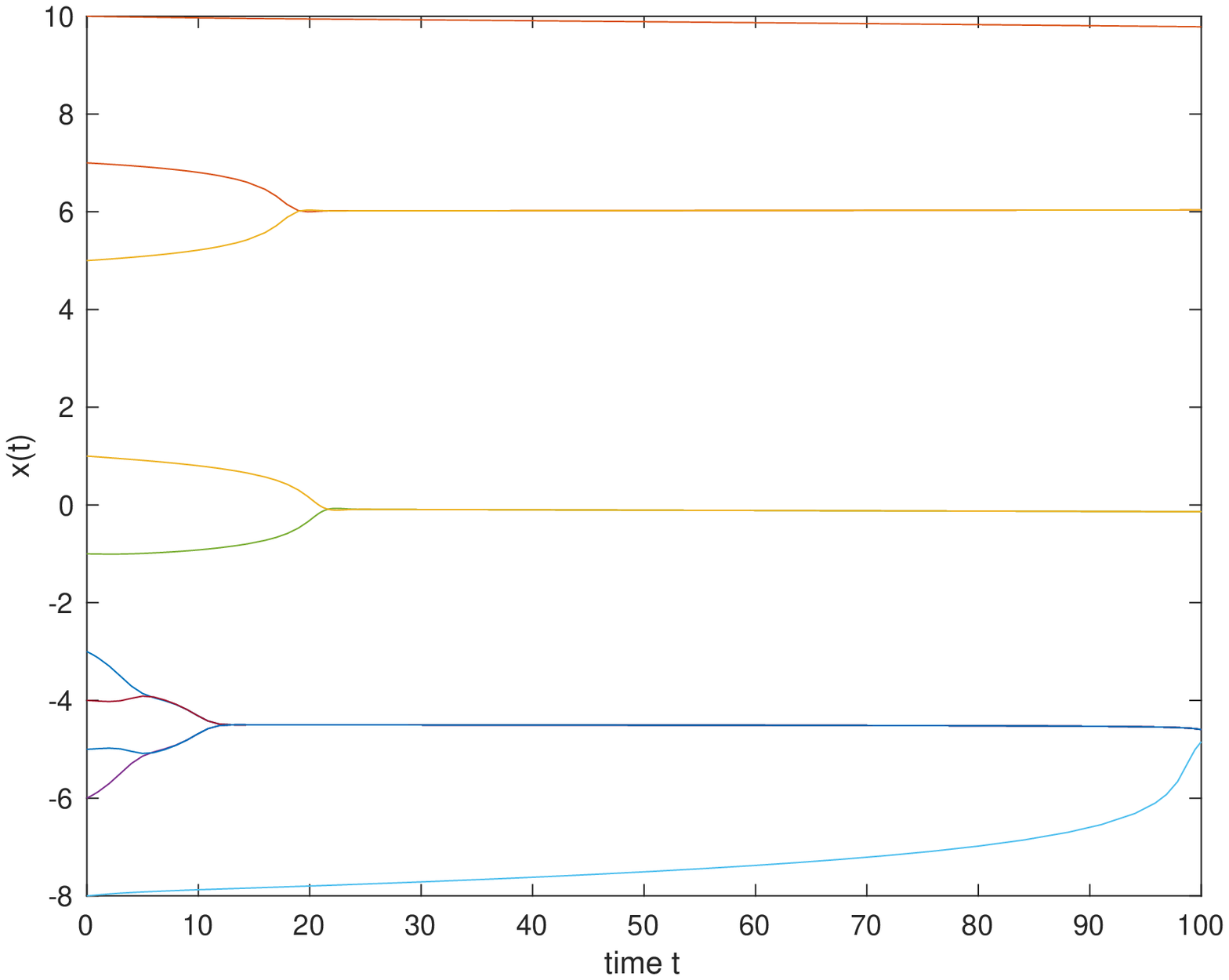}};
\node[anchor=south west] at (260,0) {\includegraphics[scale=0.38]{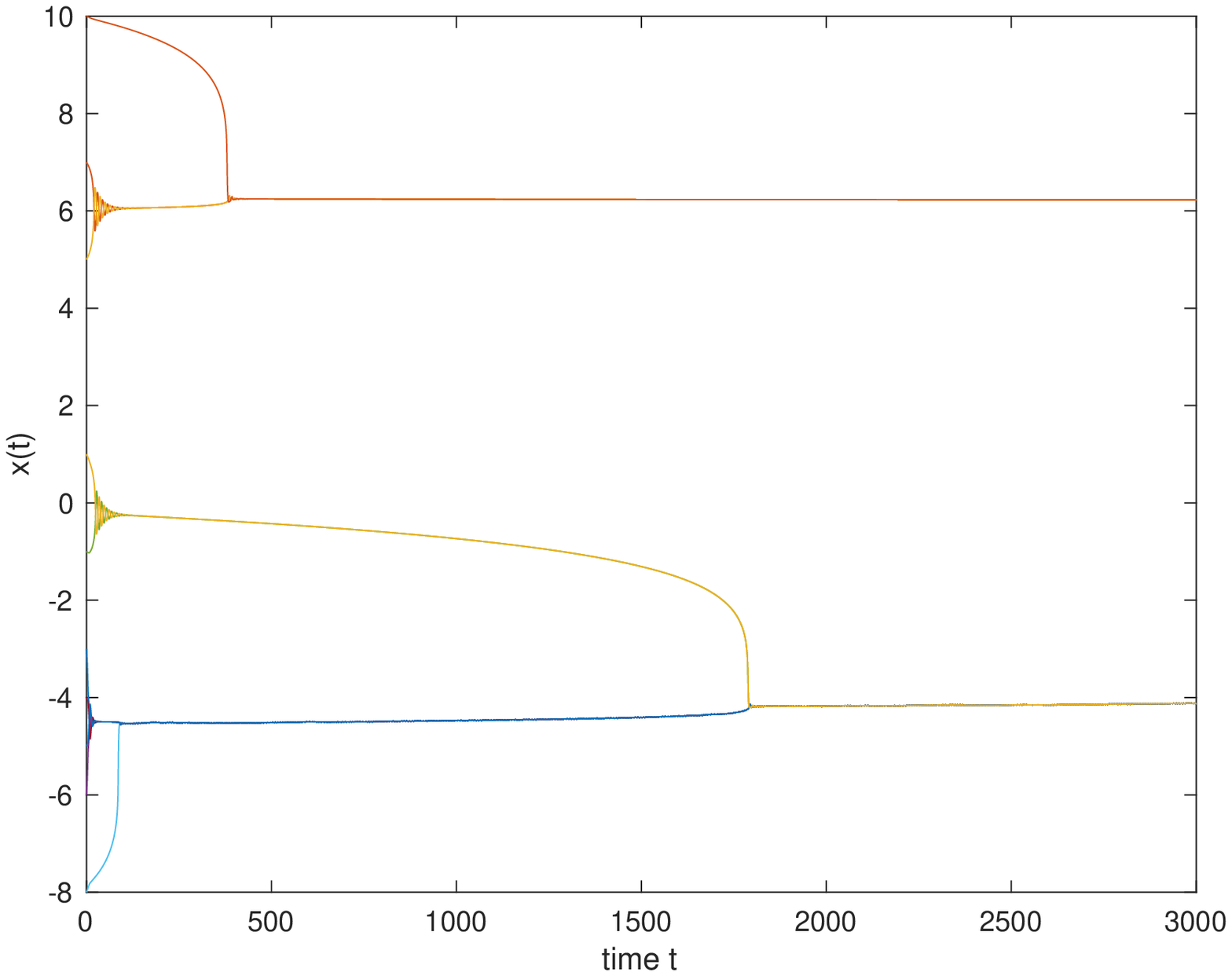}};
\node[anchor=south west] at (432,-65) {\includegraphics[scale=0.15]{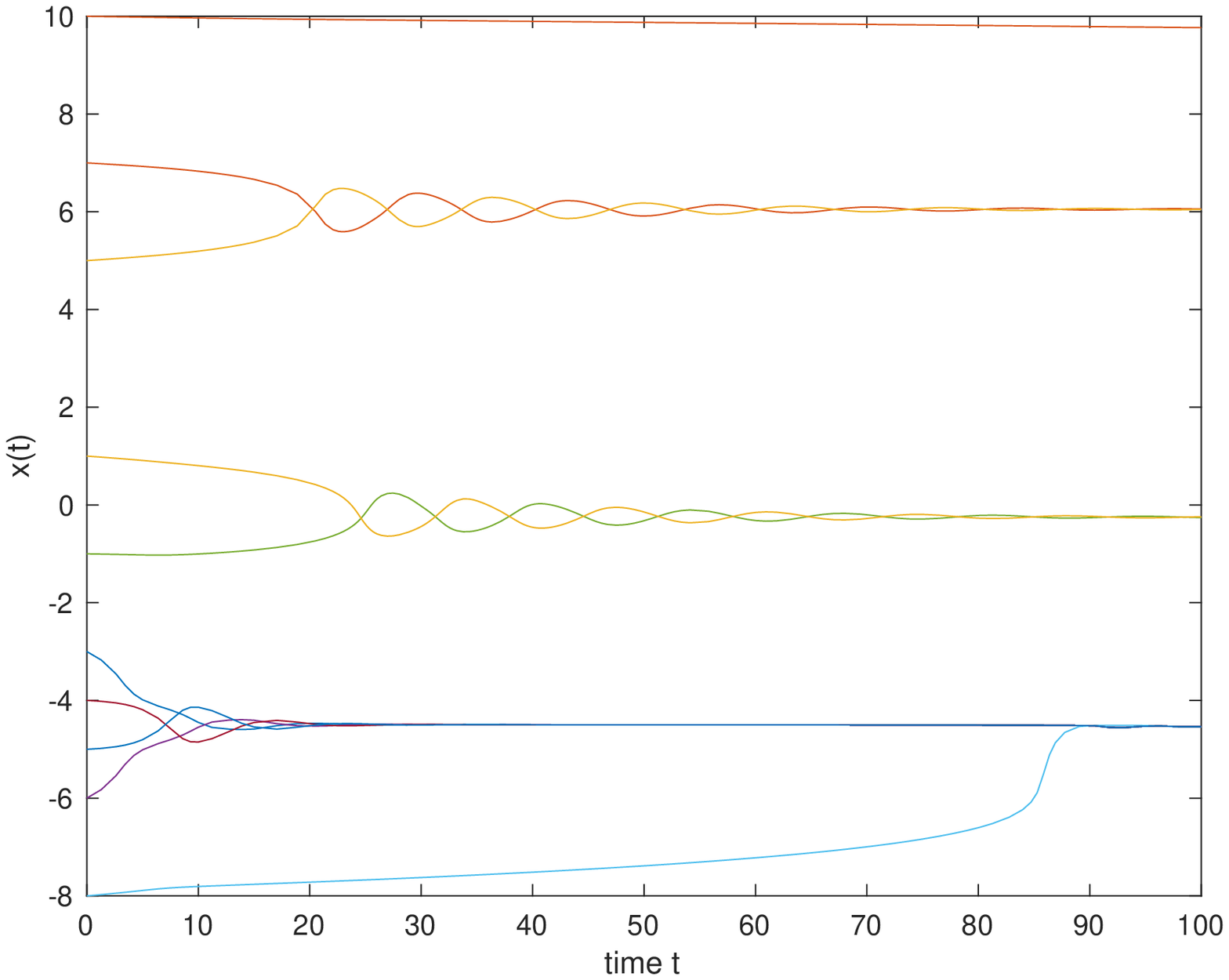}};
\node[anchor=north west] at (0,0) {\includegraphics[scale=0.38]{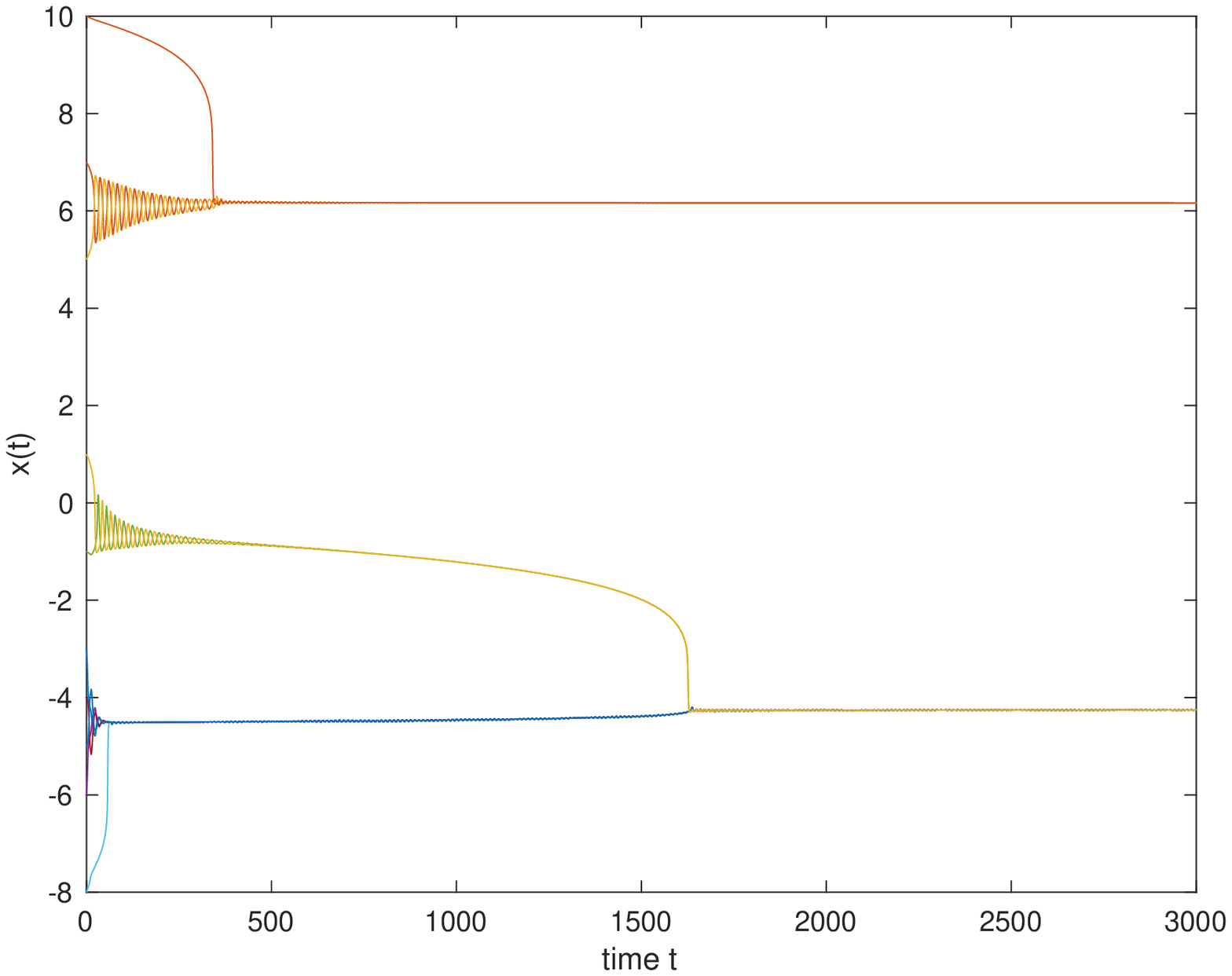}};
\node[anchor=north west] at (170,50) {\includegraphics[scale=0.15]{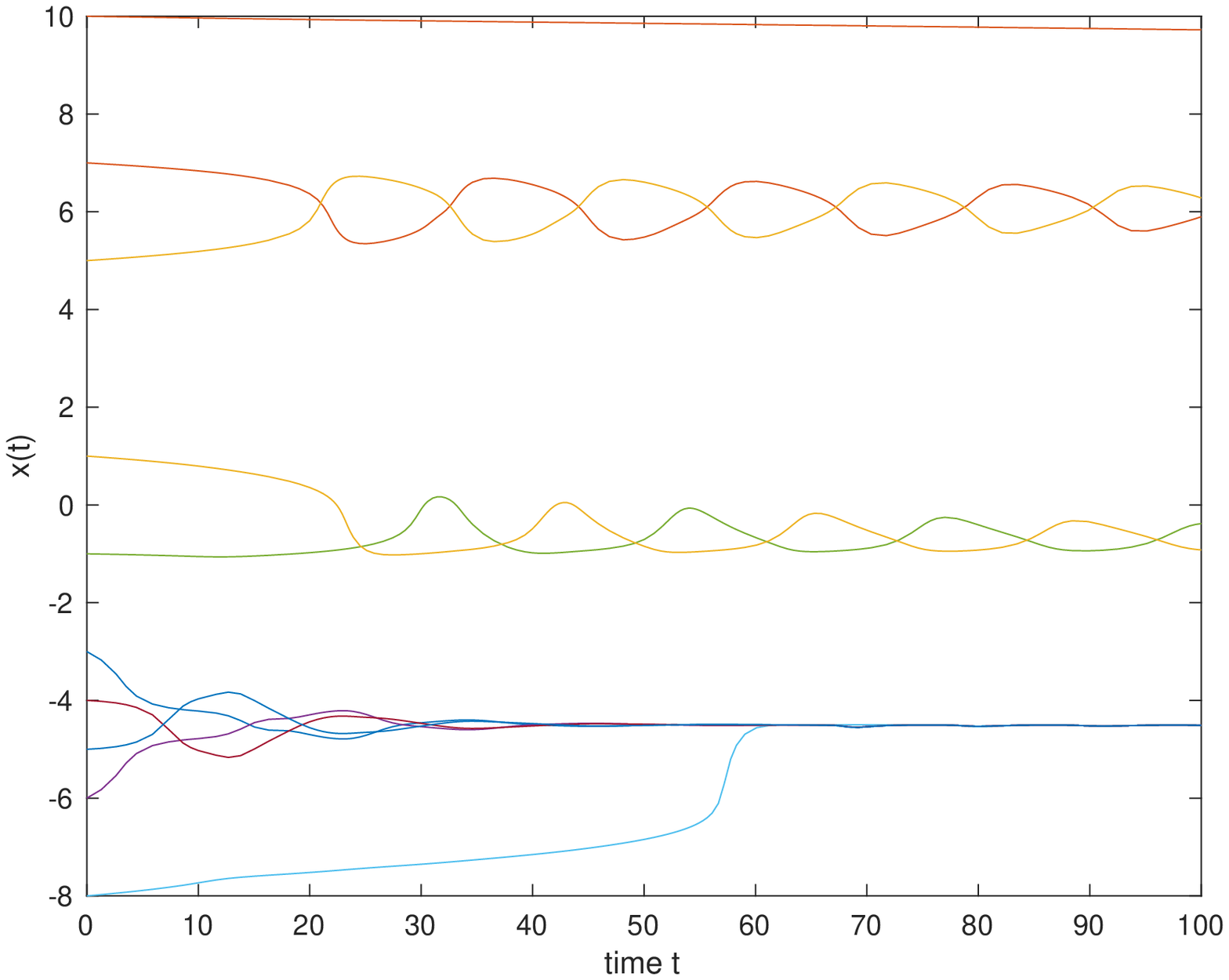}};
\node[anchor=north west] at (260,0) {\includegraphics[scale=0.38]{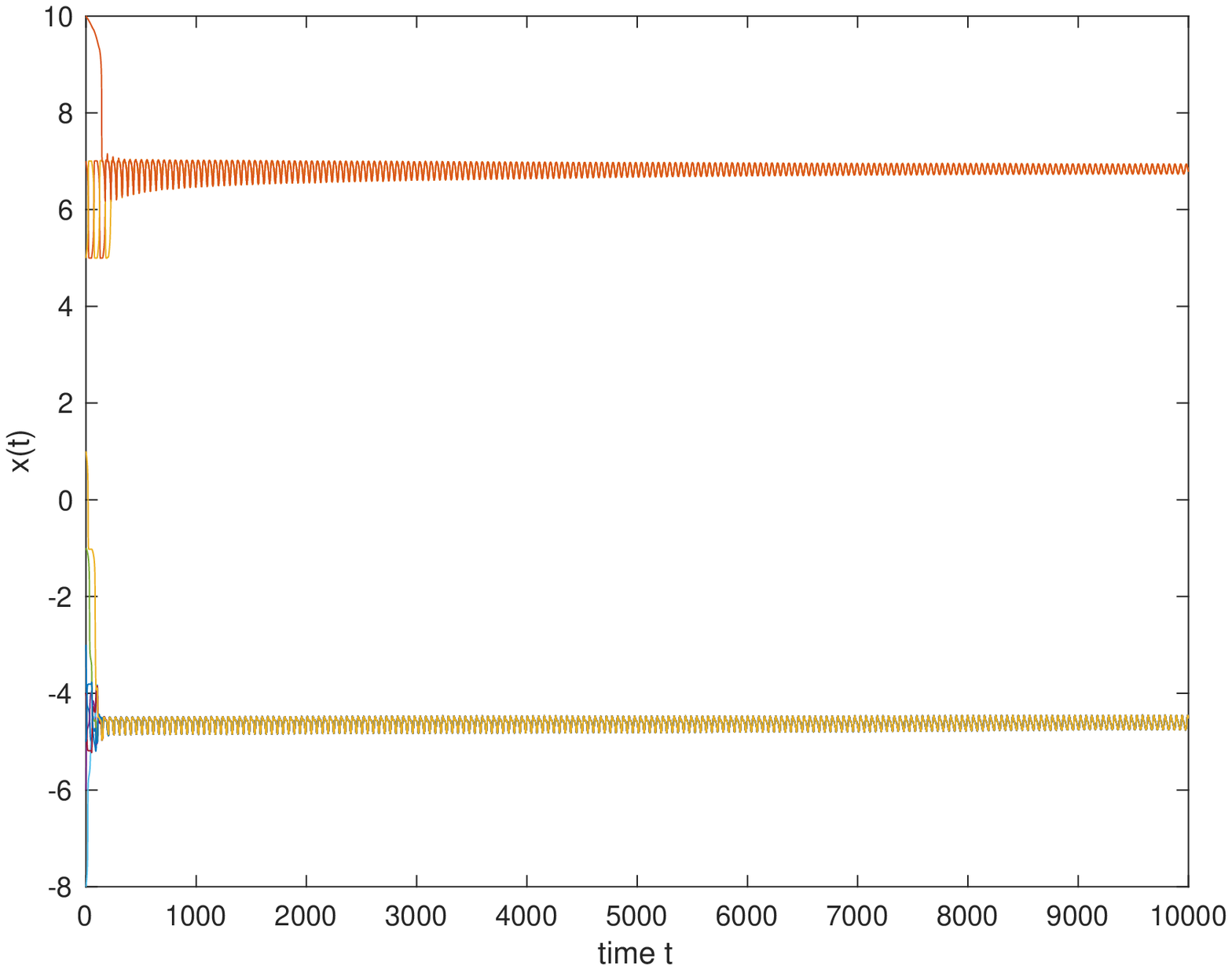}};
\node[anchor=north west] at (432,50) {\includegraphics[scale=0.15]{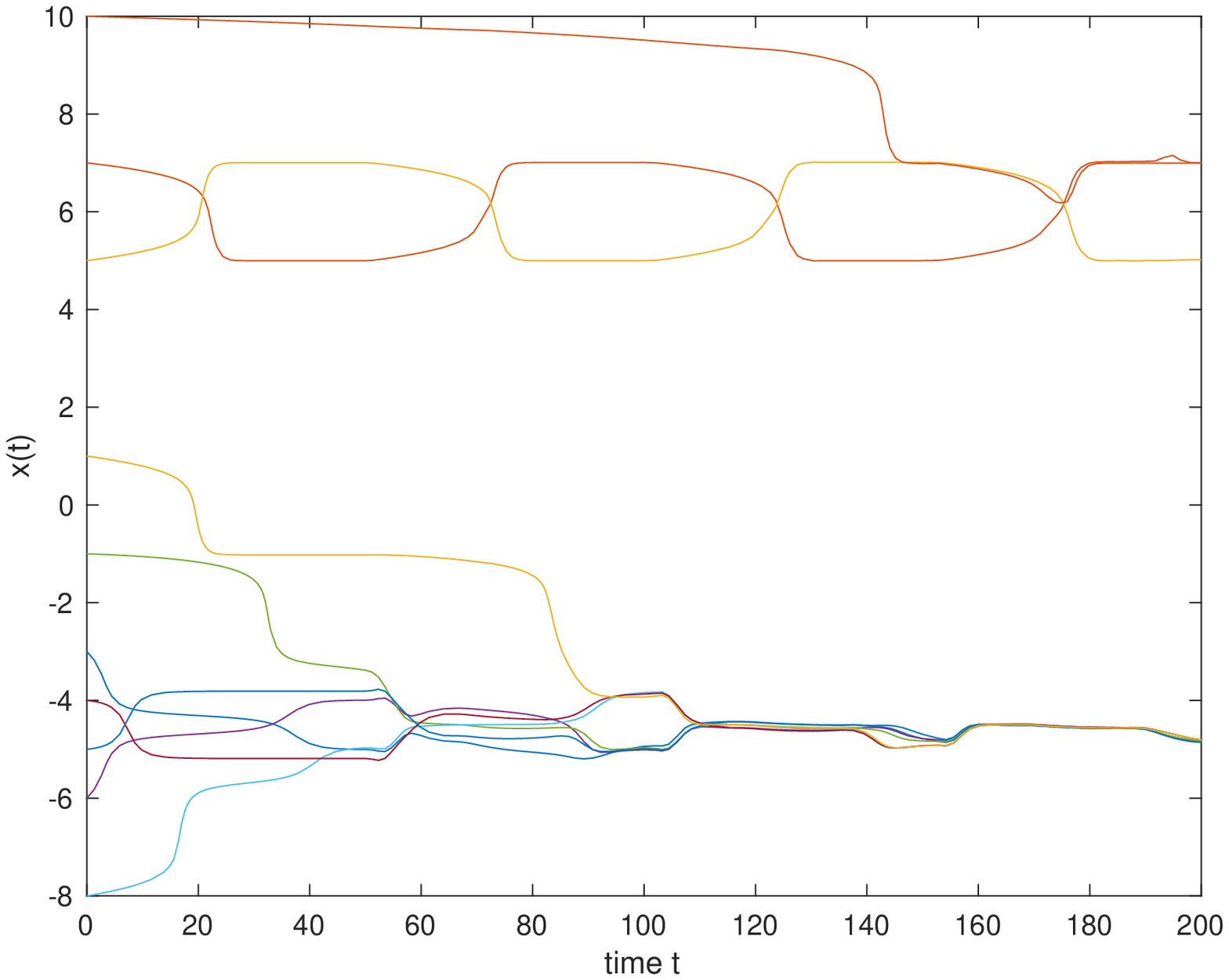}};
\end{tikzpicture}
\caption{Communication rates \eqref{notsymmetric} \& $\beta = 3$: time evolution of solutions with different strengths of time delays;  $\tau=1$ (top left), $\tau=5$ (top right), $\tau=10$ (bottom left), $\tau=50$ (bottom right).
}
\label{notsym_short}
\end{figure}

\section*{Acknowledgments}
The first author was supported by POSCO Science Fellowship of POSCO TJ Park Foundation. The second and third author were  partially supported by the GNAMPA 2019  project {\sl Modelli alle derivate parziali per sistemi multi-agente} (INdAM).

\end{document}